\documentclass{amsart}
\usepackage{tikz-cd}
\usepackage{amsmath}
\usepackage{amsfonts}
\usepackage{mathtools}
\usepackage{amsmath,amssymb,amsthm,mathrsfs}
\usepackage[utf8]{inputenc}
\usepackage{hyperref}
\usepackage{enumitem}

\theoremstyle{plain} 
\newtheorem{theorem}{Theorem}
\numberwithin{theorem}{section}
\newtheorem*{theorem*}{Theorem}
\newtheorem{prop}[theorem]{Proposition}
\newtheorem{lemma}[theorem]{Lemma}
\newtheorem{coro}[theorem]{Corollary}

\newtheorem{definition}[theorem]{Definition}
\theoremstyle{definition} 

\newtheorem*{example}{Example}
\theoremstyle{remark} 
\newtheorem{remark}[theorem]{Remark}
\theoremstyle{definition} 
\newtheorem*{ack}{Acknowledgment}

\DeclareMathOperator{\Hom}{Hom}
\newcommand{\In}{\subset}
\newcommand{\pa}{\partial}
\newcommand{\la}{\langle}
\newcommand{\ra}{\rangle}

\newcommand{\bu}{\bullet}
\newcommand{\LRA}{\Leftrightarrow}
\newcommand{\wt}{\widetilde}
\newcommand{\wb}{\overline}
\newcommand{\La}{\Lambda}
\newcommand{\RM}{\backslash}
\newcommand{\be}{\begin{equation} }
\newcommand{\ee}{\end{equation} }

\usepackage{color}

\newcommand{\bbC}{\mathbb{C}}
\newcommand{\C}{\mathbb{C}} 
\newcommand{\D}{\mathbb{D}}

\renewcommand{\P}{\mathbb{P}}
\newcommand{\Q}{\mathbb{Q}}
\newcommand{\R}{\mathbb{R}}

\newcommand{\Z}{\mathbb{Z}}

\newcommand{\cF}{\mathcal{F}}

\newcommand{\cH}{\mathcal{H}}
\newcommand{\cI}{\mathcal{I}}

\newcommand{\cK}{\mathcal{K}}
\newcommand{\cL}{\mathcal{L}}
\newcommand{\cM}{\mathcal{M}}
\newcommand{\cN}{\mathcal{N}}

\newcommand{\cV}{\mathcal{V}}

\newcommand{\bH}{\mathbf{H}}

\newcommand{\bL}{\mathbf{L}}

\newcommand{\rarr}{\longrightarrow}

\newcommand{\gr}{\textup{gr}}

\newcommand{\grf}{\gr^F_\bullet}

\newcommand{\FA}{\sA^\bullet}

\newcommand{\sL}{\mathscr{L}}

\newcommand{\sF}{\mathscr{F}}

\newcommand{\sO}{\mathscr{O}}
\newcommand{\sH}{\mathscr{H}}
\newcommand{\sA}{\mathscr{A}}
\newcommand{\sC}{\mathscr{C}}
\newcommand{\sAG}{\sA^\bullet}
\newcommand{\shTA}{\mathscr{T}}
\newcommand{\shD}{\mathscr{D}}

\newcommand{\Spec}{\textup{Spec}}
\newcommand{\Sym}{\textup{Sym}^\bullet}
\newcommand{\SiS}{\textup{SS}}
\newcommand{\DR}{\textup{DR}}
\newcommand{\MF}{\textup{MF}}
\newcommand{\DF}{\textup{DF}}
\newcommand{\KF}{\textup{KF}}
\newcommand{\Mod}{\textup{Mod}}
\newcommand{\coh}{\textup{coh}}
\newcommand{\Coker}{\textup{Coker}}
\newcommand{\Coim}{\textup{Coim }}
\newcommand{\Ker}{\textup{Ker}}
\newcommand{\an}{\textup{an}}
\newcommand{\ch}{\textup{ch}}
\newcommand{\td}{\textup{td}}

\newcommand{\Kos}{\textup{Kos}}

\newcommand{\supp}{\textup{supp}}
\newcommand{\gdeg}{\textup{gdeg}}
\newcommand{\bs}{{\bf s}}

\newcommand{\bh}{{\bf h}}

\newcommand{\bq}{{\bf q}}
\newcommand{\bff}{{\bf f}}
\newcommand{\lup}[1]{{}^#1}
\newcommand{\vb}{{\bf v}}
\newcommand{\rel}{\textup{rel}}

\begin{document}
\title{Log $\shD$-modules and index theorems}

\author{Lei Wu}
\address{Lei Wu, Department of Mathematics, University of Utah.
155 S 1400 E,  Salt Lake City, UT 84112, USA}
\email{lwu@math.utah.edu}


\author{Peng Zhou}
\address{Peng Zhou, Institut des Hautes \'Etudes Scientifiques. Le Bois-Marie, 35 route de Chartres, 91440 Bures-sur-Yvette France}
\email{pzhou.math@gmail.com}


\begin{abstract}
    We study log $\shD$-modules on smooth log pairs and construct a comparison theorem of log de Rham complexes. The proof uses Sabbah's generalized $b$-functions. As applications, we deduce a log index theorem and a Riemann-Roch type formula for perverse sheaves on smooth quasi-projective varieties. The log index theorem naturally generalizes the Dubson-Kashiwara index theorem on smooth projective varieties. 
\end{abstract}

\maketitle

\section{Introduction} 
Let $(X,D)$ be a smooth log pair, that is, $X$ is a smooth variety over $\bbC$, and $D$ is a reduced normal crossing divisor. Denote the open embedding by $j\colon U=X\setminus D\hookrightarrow X$. The sheaf of log differential operators $\shD_{X, D}$ is the sub-sheaf of $\shD_X$ consisting of differential operators that preserve the defining ideal of the divisor. Log $\shD$-modules are (left or right) modules over $\shD_{X,D}$. In this article, we mainly focus on studying log $\shD$-modules associated to $\shD_U$-modules, called lattices (Definition \ref{def:lattice}), and we use them to study perverse sheaves on $U$.

\subsection{A comparison theorem for log $\shD$-modules}
Our first theorem is a comparison between a $\shD_U$-module $\cM_U$ and its lattice $\cM$ in terms of (log) \emph{de Rham} (DR) complexes (see \eqref{eq:logdrd} for its definition). 

\begin{theorem}[Log Comparison]\label{thm:logcom}
Let $\cM_U$ be a regular holonomic $\shD_U$-module. Then for every lattice $\cM$ of $\cM_U$, there exists $q_0>0$ so that for all $q> q_0$ we have natural quasi-isomorphisms
\begin{equation}\label{eq:logrh}
 \DR_{X,D}(\cM(qD))\stackrel{q.i.}{\simeq} Rj_*\DR_U(\cM_{U}).
\end{equation}
\end{theorem}

Suppose $D_i$, $i=1, \cdots, k$, are irreducible components of $D$. Then we may consider the following lattices
\[ \cM( a_1 D_1 + \cdots + a_k D_k), \quad a_i \in \Z \]
and the corresponding log DR complexes, 
\begin{equation}\label{eq:question}
    \DR_{X, D}(\cM(a_1 D_1 + \cdots + a_k D_k)). 
\end{equation}
It is interesting to ask: how will the above log DR complex change as we vary the shift parameter $a_i$? In this perspective, our log comparison theorem only says  the log de Rham complex stabilizes if all the $a_i$ are large enough.

\subsubsection{Application to Deligne Lattices}
We compute \eqref{eq:question} in the case of flat connections with regular singularities along boundaries. Assume that $\overline\cV$ is the Deligne lattice of a local system $L$ on $U$ satisfying that real parts of eigenvalues of the residues along $D$ are in $(-1,0]$; see \S \ref{sec:dl} for details.
For simplicity, we write $D=\sum_{l=1}^kD_l$,  
\[D^{I}=\sum_{l\in I}D_l \quad \textup{ and } \quad D^{\bar I}=\sum_{l\notin I}D_l\]
for a subset $I\subseteq \{1,2,\dots,k\}$ and 
$$j^I_1: U\to X\setminus D^I \textup{ and }j^I_2: X\setminus D^I\to X.$$
In this case, Theorem \ref{thm:logcom} can be refined as follows. 
\begin{theorem}\label{thm:DLlc}
For every pair of integers $k_1,k_2>0$, we have quasi-isomorphisms
\[\DR_{X,D}(\overline{\cV}(k_2D^{\bar I}-k_1D^{I}))\stackrel{q.i.}{\simeq}\DR_{X,D}(\overline{\cV}(-D^{I}))\stackrel{q.i.}{\simeq}j^I_{2!}Rj^I_{1*}L[n].\]
In particular, 
 \[\DR_{X,D}(\sO_X(k_2D^{\bar I}-k_1D^{I}))\stackrel{q.i.}{\simeq}\DR_{X,D}(\sO_X(-D^{I}))\stackrel{q.i.}{\simeq}j^I_{2!}Rj^I_{1*}\bbC_U[n].\]
\end{theorem}

\subsection{Application to $\shD_Y[\bs](\bh^{\bs+\vb}\cdot \cM_0)$}
Suppose that $\bh=(h_1,\dots,h_k)$ is a $k$-tuple of regular functions on a smooth variety $Y$ of dimension $m$. 
Let $U=Y\setminus (\prod_lh_l=0)$, and  $j: U\hookrightarrow Y$ \footnote{Here we deviate from the log pair setup and notation. Later the log pair $(X, D)$ shows up from graph embedding.}.  For a regular holonomic $\shD_{U}$-module $\cM_U$, we consider the $\shD_Y[\bs]$-module for $\vb=(v_1,\dots,v_l)\in \Z^k$:
\[\cM^\vb_\bh\coloneqq\shD_Y[\bs](\bh^{\bs+\vb}\cdot \cM_0)\subseteq j_*(\bh^{\bs}\cdot\cM_U[\bs])\]
where $\bh^{\bs+\vb}=\prod_lh_l^{s_l+v_l}$ and $\bs=(s_1,\dots,s_k)$ are independent variables and $\cM_0$ is a coherent $\sO_X$-submodule of $j_*\cM_U$ that generates $j_*\cM_U$ over $\shD_Y$.

Following Ginsburg's ideas and using Sabbah's generalized Bernstein-Sato polynomials, we obtain the following generalization of the Beilinson Theorem \cite[Proposition 3.6]{Gil} (see also \cite{BG}):
\begin{theorem}\label{thm:introex1}
For $\vb=(v_1,v_2,\dots,v_k)\in \Z^k$ with $v_l\gg 0$ for every $l$, we have  
\begin{enumerate}[label=\textup{(\roman*)}]
    \item $j_* \cM_U = \shD_Y[\bs](\bh^{\bs-\vb}\cdot \cM_0) / (s_1, \cdots, s_n) \shD_Y[\bs](\bh^{\bs-\vb}\cdot \cM_0) $;\label{item:intro111}
    \item  $j_! \cM_U = \shD_Y[\bs](\bh^{\bs+\vb}\cdot \cM_0) / (s_1, \cdots, s_n) \shD_Y[\bs](\bh^{\bs+\vb}\cdot \cM_0) $.\label{item:intro222}
\end{enumerate}
\end{theorem}

There is also a similar but different result where we set $\bs=0$ from the beginning, instead of taking the fiber at $\bs=0$ as above. 
\begin{theorem}\label{thm:introex2}
For $\vb=(v_1,v_2,\dots,v_k)\in \Z^k$ with $v_l\gg 0$ for every $l$, we have  
\begin{enumerate}[label=\textup{(\roman*)}]
    \item $j_* \cM_U = \shD_Y(\bh^{-\vb}\cdot \cM_0)  $;\label{jstar}
    \item  $j_{!*} \cM_U = \shD_Y(\bh^{\vb}\cdot \cM_0)   $.\label{jshstar}
\end{enumerate}
\end{theorem}
\noindent This should be known to experts and is a small generalization to Budur's result \cite[Theorem 5.2]{Budur}.

The $\shD_{Y}[\bs]$-modules 
 \[ \cM_\bh\coloneqq D_Y[\bs](\bh^\bs \cdot \cM_0) \textup{ and } j_*(\bh^{\bs}\cdot\cM_U[\bs]) \]
can be alternatively understood as two families of $\shD_Y$-modules over $\Spec \bbC[\bs]$. Theorem \ref{thm:introex1} is about their fibers at integral $\bs \gg 0$ and $\bs \ll 0$ (i.e. each $s_i \gg 0$ or $s_i \ll 0$ respectively). The proofs of Theorem \ref{thm:introex1} and Theorem \ref{thm:introex2} use localization over the generic point and the geometric points in $\Spec \bbC[\bs]$; see \S \ref{sec:bbloc} for details. The inclusion in family:
\[\cM_\bh\hookrightarrow j_*(\bh^{\bs}\cdot\cM_U[\bs])\]
induces fiberwise morphisms of (complexes of) $\shD_Y$-modules
\[\iota_{\bf a}: \bbC_{\bf a}\stackrel{\bL}\otimes_{\bbC[\bs]}\cM_\bh \to \bbC_{\bf a}\stackrel{\bL}\otimes_{\bbC[\bs]}j_*(\bh^{\bs}\cdot\cM_U[\bs])\simeq j_*(\bh^{\bf a}\cdot\cM_U[\bs])\]
at closed points ${\bf a}=(a_1,a_2,\dots,a_k)\in \Spec \bbC[\bs]$, where $\bh^{\bf a}=\prod_lh_l^{a_l}$ and $ \bbC_{\bf a} =  \bbC[\bs]/(\bs - \bf a)$. The $\shD_Y$-modules $j_*(\bh^{\bf a}\cdot\cM_U)$ are $\Z^k$-periodic in ${\bf a}$ and $j_*(\bh^{\bf k}\cdot\cM_U)=j_*(\cM_U)$ for ${\bf k}\in \Z^k$. Here $\bh^{\bf a}\cdot\cM_U$ is the $\shD_U$-module of $\cM_U$ twisted by the local system given by $\bh^{\bf a}$.

Theorem \ref{thm:introex1}\ref{jstar} (or more precisely, Corollary \ref{cor:mainc}) implies that 
\be \label{eq:aaaaaa}
\iota_{\bf a} \textup{ is a quasi-isomorphism for all }a_l\gg0.
\ee
In general, we have that 
\be \textup{Image}(\iota_{\bf a})=\shD_Y(\bh^{\bf a}\cdot\cM_0)\subseteq j_*(\bh^{\bf a}\cdot\cM_U)
\ee
for every ${\bf a}\in \bbC^k$.

\begin{remark}
In the case that $\cM_U=\sO_U$, the result \eqref{eq:aaaaaa} has an application in \cite{BVWZ} to prove a conjecture of Budur \cite{Budur} about zero loci of Bernstein-Sato ideals, which generalizes a classical theorem of Malgrange and Kashiwara relating the $b$-function of a multivariate polynomial with the monodromy eigenvalues on the Milnor fibers cohomology.
\end{remark}


Consider the graph embedding of $\bh$:
\[\eta^\bh\colon Y\hookrightarrow X:=Y\times \bbC^k,  \quad y\mapsto (y, h_1(y),\dots,h_k(y)).\]
Let $t_1, \cdots, t_k$ be coordinates of $\C^k$, and $ D = \bigcup_{i=1}^k \{t_i=0\}$ be the boundary divisor on $X$. Following Malgrange \cite{MalV}, we have a canonical isomorphism 
\[\eta^\bh_{+}(j_*\cM_U)\simeq j_*(\bh^{\bs}\cdot\cM_U[\bs]).\]
Then $\cM^{\vb}_\bh$ is a $\shD_{X,D}$-lattice of $\eta^\bh_{+}(j_*\cM_U)$ for every $\vb\in \Z^k$. See \S\ref{sec:bbloc} and \S\ref{sect:crf} for more details.


By \eqref{eq:drkos}, as we identify $s_l$ with $-t_l\partial_{t_l}$, we have 
\[\eta^\bh_*(\DR(\bbC\stackrel{\bL}{\otimes}_{\bbC[\bs]}\cM^{\vb}_\bh))\simeq \DR_{X,D}(\cM^{\vb}_\bh),\]
Using Corollary \ref{cor:mainc}, we hence have for all $v_l\gg0$
\[\DR_{X,D}(\cM^{-\vb}_\bh)=\eta^\bh_*(Rj_*\DR(\cM_U)),\]
which recovers Theorem \ref{thm:logcom} in this case for lattices of $\eta^\bh_{+}(j_*\cM_U)$. Since for $(a_1,a_2,\dots,a_k)\in \Z^k$
\[
\DR_{X,D}(\cM_\bh^{\vb+(a_1,a_2,\dots,a_k)})=\DR_{X,D}(\cM_\bh^{\vb}(-a_1D_1-\cdots -a_kD_k)),
\]
Corollary \ref{cor:mainc} further implies 
\[\DR_{X,D}(\cM_\bh^{\vb}(-a_1D_1-\cdots -a_kD_k))\simeq \eta^\bh_!(j_!\DR(\cM_U))\]
for a fixed $\vb \in \Z^k$ and for all integral $a_l$ large enough, which gives a partial answer to the question of \eqref{eq:question} for lattices $\cM_{\bh}^\vb$.

\subsection{Index theorems in the log cotangent bundle}
One application of the log comparison theorem is to find a formula for the Euler characteristic of a perverse sheaf $\sF^\bullet$ on $U$:
$$\chi(U, \sF^\bullet)=\sum_{i}(-1)^i\dim \bH^i(U, \sF^\bullet).$$ 
If $U=X$ is compact, then we have the index theorem of Dubson-Kashiwara (\cite{Dub} and \cite{Kash-index}): 
\begin{equation}\label{eq:kdind}
    \chi(X, \sF^\bullet)=[\SiS\sF^\bullet] \cdot [T^*_XX]
\end{equation}
where $T^*_XX$ is the zero section of $T^*X$.  For open $U$, Kashiwara shows 
\be \chi(U, \sF^\bullet)=[\SiS\sF^\bullet] \cdot \Gamma_{d\varphi},\ee
where $\Gamma_{d \varphi}$ is certain perturbation of $T_U^* U$; see Theorem \ref{thm:Kasind}.

In this paper, we compactify $T^*U$ to the log cotangent bundle $T^*(X, D)$, and we get:
\begin{theorem}\label{thm:logind}
Let $(X, D)$ be a smooth log pair with $X$ projective and $U = X \RM D$, and let $\sF^\bullet$ be a perverse sheaf on $U$. Then 
\be \label{eq:logie}
 \chi(U, \sF^\bullet)= [\overline{\SiS}\sF^\bullet] \cdot [T^*_X(X, D)]
\ee
where $\overline{\SiS}\sF^\bullet$ is the closure of $\SiS\sF^\bullet \In T^*U$ inside $T^*(X, D)$, and $T^*_X(X,D)$ is the zero section of $T^*(X,D)$.
\end{theorem}

Since $[\overline\SiS\sF^\bullet]$ and $[T^*_X(X,D)]$ are two $n$-cycles in the Chow ring of the logarithmic cotangent bundle $T^*(X,D)$, their intersection gives a 0-cycle in $X\simeq T^*_X(X,D)$ and  $[\overline\SiS\sF^\bullet]\cdot[T^*_X(X,D)]$ simply means the degree of the $0$-cycle.
To compute the intersection on the RHS, we can try to perturb the zero-section from $T^*_X(X, D)$ to a smooth ($C^\infty$) section, that intersects $\overline{\SiS}\sF^\bullet$ transversally, then one needs to count  intersection points with sign. If $T^*_X(X, D)$ can be perturbed as a holomorphic section that intersects $[\overline\SiS\sF^\bullet]$ transversally, then one only needs to count unsigned intersection points.

If $\sF^\bullet=\bbC_U[n]$ for $\dim U=n$, then  \eqref{eq:logie} implies
\be\chi(U)=(-1)^n[T^*_X(X, D)] \cdot [T^*_X(X, D)],\ee
for every normal crossing compactification $(X,D)$ of $U$.

\begin{theorem}[Non-compact Riemann-Roch]\label{thm:openRR}
In the situation of Theorem \ref{thm:logind}, let $\sF^\bullet$ be a perverse sheaf on $U$ and ${\SiS \sF^\bullet} = \sum_v n_v \La_v$.  Assume that there exists a rational function $f$ on $X$, such that $f|_U$ is non-vanishing, and $\Gamma_{d \log f}$ intersects each $\wb \La_v$ transversally in $T^*(X, D)$ with intersection number denoted as $\gdeg_f^{\log}(\wb \Lambda_v)$,  then 
\[\chi(U, \cF^\bullet)=\sum_{v}n_v \cdot \gdeg_f^{\log}(\wb \Lambda_v).\]
In particular $\chi(U, \cF^\bullet) \ge 0$. 
\end{theorem}

\begin{remark}
If $U = (\C^*)^n$ with coordinates $(z_1,\dots,z_n)$, Franecki and Kapranov proved
\[\chi((\C^*)^n, \sF^\bullet) = \sum_v n_v \gdeg_f(\Lambda_v), \]
where $\gdeg_f(\Lambda_v)$ is the intersection of $\Gamma_{d \log f}$ with $\Lambda_v$ in $T^*U$ for $f=z_1^{s_1} \cdots z_n^{s_n}$ with generic non-zero $s_i \in \Z$. Using it, they further proved a non-compact Riemann-Roch theorem on quasi-abelian varieties \cite[Theorem (1.3)]{FK}. 

Take the natural compactification $(\C^*)^n\hookrightarrow \P^n$ and set $D$ to be the boundary divisor. Note that by generic choice of $f$, one ensures the intersection of $\Gamma_{d \log f}$ with $\wb \La_v$ in $T^*(\P^n, D)$ is actually contained in $T^*U$, hence $\gdeg_f(\Lambda_v) = \gdeg_f^{\log}(\wb \La_v)$. Therefore, applying Theorem \ref{thm:openRR} gives an alternative proof of Theorem (1.3) in \emph{loc. cit.}
\end{remark}

\begin{example}
Let $X = \P^2$ with homogeneous coordinates $(X,Y,Z)$, $D = (XYZ=0)$ and $U = (\C^*)^2$. Let $x = X/Z, y=Y/Z$ be affine coordinates on $U$.  Let  $S = \{(x,y) \in U \mid y=x(1-x)\}$, then $S$ is $\P^1$ removing three points. Let $L = \C_S[1]$ where the degree shift is to make it a perverse sheaf. We then have 
\[ \chi(U, L) = -(1 -2) = 1.\]
On the other hand, we can take a rational function on $X$ as $f = X/Y$, then $f|_U = x/y$ is non-vanishing. Then on the $U_0 \cong \C^2$ patch where $Z\neq 0$, we have 
\[\Gamma_{d \log f} = \frac{dx}{x} - \frac{dy}{y}. \]
Then $\SiS L$ in $T^*U$ is the conormal bundle of $S$, which is $\C$-cone generated by $d(y-x(1-x)) = (-1+2x)dx + dy$, thus
\[ \SiS L  = \{(x, y; \xi_x, \xi_y) \in T^*U \mid y - x(1-x)=0, \xi_x  = (-1+2x) \xi_y \}. \]
We can verify that $\Gamma_{d \log f} \cap \SiS L \cap T^*U = \emptyset$. Thus, one only need to check the intersection of $\Gamma_{d \log f} \cap \SiS L$ in $T^*(X, D)$ over the boundary $D$. 

Then the equations for $\wb{\SiS} L$ in  $T^*(X, D)|_{U_0}$ are
\[ y=x(1-x), \quad (1-x)\eta_x = (-1+2x) \eta_y. \]
Hence, there is one intersection point at $(x,y)=(0,0)$ and $(\eta_x, \eta_y)=(1,-1)$. Similarly, one can check in other patches $U_1\{X \neq 0\}$ and $U_2 = \{Y\neq0\}$, to see there is no additional intersections. Thus 
\[ \#(\wb{\SiS }L \cap \Gamma_{d \log f}) = 1.\]
This verifies the non-compact Riemann-Roch theorem. 
\end{example}

\subsection{Sketch of the proof of main theorems}
For the Log Comparison theorem, we have to understand the behavior of log de Rham complexes for lattices $\cM(qD)$ as $q$ changes. To this purpose, we introduce Bernstein-Sato polynomials (or $b$-functions) for lattices. By applying Sabbah's multi-filtrations for holonomic modules \cite{Sab}, the existence of $b$-functions for lattices is guaranteed by Theorem \ref{thm:bflattice}. Using $b$-functions for lattices, we show that $\DR_{X,D}(\cM(qD))$ stabilizes for large $q$ (actually for large $|q|$ is enough). Then we note that taking inductive limit as $q\to \infty$ 
$$\lim_{q\to \infty}\DR_{X,D}(\cM(qD)) = \DR_{X,D}(\lim_{q\to \infty} \cM(qD)) = \DR(j_*\cM_U).$$ 
Hence for large enough $q$, we have $\DR_D(\cM(qD)) = \DR(j_*\cM_U)$. 

The (log) \emph{de Rham} functors in the log comparison theorem mean the analytic ones by GAGA, as perverse sheaves are defined in the Euclidean topology. One can replace the algebraic regular holonomic module $\cM_U$ by $\widetilde\cM(*D)$ and it still holds in the analytic category, where $\widetilde\cM$ is an analytic regular holonomic $\shD_X$-module and $\widetilde\cM(*D)$ is its algebraic localization along $D$ (see \cite[Chapter II.5]{Bj} for definitions).
 \vskip 5mm
 
For the index theorem, we provide two proofs. 

    (1) The first one is more algebraic and motivated by Laumon's algebraic approach to the Dubson-Kashiwara index equality \eqref{eq:kdind} \cite[\S 6]{Laumon}. Roughly speaking, adapting Laumon's construction of direct images of filtered $\shD$-modules \cite{Laumon1} to the log situation, we obtain a Laumon-type direct-image formula for log $\shD$-modules; see Theorem \ref{thm:Laumf} and Corollary \ref{cor:imageK}.  Using the Laumon-type formula and Grothendieck-Riemann-Roch together with Ginsburg's characteristic cycle theorem (Theorem \ref{thm:Gil}), we first obtain an index result for lattices; see Theorem \ref{thm:GIT}. It would further imply the index equality \eqref{eq:logie} after applying Theorem \ref{thm:logcom}.
    \vskip 3mm
    
    (2) The second proof is more topological, and it is presented at Section 8. The Euler characteristic of $\chi(\Hom(F, G))$ for two constructible sheaves $F,G$ on an open $U$, can be computed by counting intersection of $\SiS(F)$ and $\SiS(G)$ (\cite{Gil}), or one can directly compute $\Hom(F, G)$ as $\Hom_{\textup{Fuk}(T^*U)}(\cL(F), \cL(G))$ where $\cL(F), \cL(G)$ are Lagrangians in Fukaya category \cite{NZ}. The first step in considering Lagrangian intersection to make them intersect transversally. If both Lagrangians are compact, then any Hamiltonian perturbation is allowed, if two Lagrangians intersects at 'infinity', then certain directed perturbation ('wrapping') near infinity is needed. In our case, one of the Lagrangian is $T^*_U U$, and we can perturb it to the Lagrangian section (\cite{Kash-index}) $\Gamma_{d \varphi}$, where $\varphi:U \to \R$ has some growth property near infinity. Our only task, is then choose a nice perturbation, so that $\Gamma_{d \varphi} \in T^*U$ extends to a smooth section $\overline\Gamma_{d \varphi}$ in the log cotangent bundle $T^*(X, D)$. A prototypical example is $d \log |z|$ as a smooth section in the log cotangent bundle $T^*(\P^1, \{0, \infty\})$. Also, one need to check there is no additional contribution to the Lagrangian intersection in replacing $\Gamma_{d \varphi}$ by its log closure $\overline\Gamma_{d \varphi}$. 
    
    At last, the topological proof of Theorem \ref{thm:logind} makes it still hold when replacing $X$ by a complex compact manifold. 

\subsection{Outline of the paper}
In \S \ref{sec:pre}, we collect basic properties for log $\shD$-modules. We discuss the direct images of log $\shD$-modules under proper morphisms of log smooth pairs in \S\ref{sec:directimage}. \S \ref{sec:bfs} is about $b$-functions for lattices and the proof of the log comparison theorem. In \S\ref{sec:bbloc}, we discuss $\shD$-modules constructed from graph embeddings from a log point of views. In \S\ref{sec:logind}, \ref{sec:noncrr} and \ref{sec:rrtop}, we discuss the index theorem in the logarithmic cotangent bundle and non-compact Riemann-Roch.   

\begin{ack}
The authors specially thank Botong Wang; the paper was partially motivated by a Riemann-Roch type question asked by him. They also thank Victor Ginsburg for answering their questions and Nero Budur, Clemens Koppensteiner and Mihnea Popa for useful comments on the first draft of the paper.  PZ is supported by the Simons Postdoctoral Fellowship as part of the Simons Collaboration on HMS, and he would like to thank IHES for providing an excellent working environment. 
\end{ack}

\section{Preliminaries on $\shD_{X,D}$-modules}\label{sec:pre}
Let $(X,D)$ be a log smooth pair with $\dim X=n$. Set $\shTA_X(-\log D)$ to be the locally free subsheaf of $\shTA_X$ generated by algebraic vector fields with logarithmic zeros along $D$, and $\shD_{X,D}$ the subalgebra of $\shD_X$ generated by $\sO_X$ and $\shTA_X(-\log D)$. One can check that $\shD_{X,D}$ is a coherent and noetherian subalgebra of $\shD_X$. We then consider the logarithmic cotangent bundle
\[T^*(X,D)\coloneqq \Spec\sA^\bullet_{X,D},\]
where $\sA^\bullet_{X,D}=\Sym(\shTA_X(-\log D))$, the symmetric algebra of $\shTA_X(-\log D)$.
Then the natural inclusion 
$$\shTA_X(-\log D)\hookrightarrow \shTA_X$$ 
induces a morphism over $X$
\[\begin{tikzcd}
T^*X\arrow[rr]\arrow[dr]& & T^*(X,D)\arrow[dl, "\pi"]\\
&X&.
\end{tikzcd}\]
We use $T^*_X(X,D)$ to denote the zero section of $T^*(X,D)$. 

We consider both left and right $\shD_{X,D}$-modules. The following lemma is useful and easy to check. 
\begin{lemma}\label{lm:pchrules}
Let $\cM$ and $\cM'$ be two left $\shD_{X,D}$-modules and let $\cN$ and $\cN'$ be two right $\shD_{X,D}$-modules. Then 
\[\cM\otimes_\sO\cM'\textup{, } \cH\textup{om}_\sO(\cM, \cM') \textup{ and }  \cH\textup{om}_\sO(\cN, \cN') \]
are left $\shD_{X,D}$-modules, and
\[\cM\otimes_\sO\cN \textup{ and }  \cH\textup{om}_\sO(\cM, \cN) \]
are right $\shD_{X,D}$-modules, where the $\shD_{X,D}$-module structures are induced by product rules or chain rules of taking differentiations.
\end{lemma}
The order filtration of $\shD_X$ induces an order filtration $F_\bullet$ of $\shD_{X,D}$ with the associated graded algebra $\grf\shD_{X,D}$ and we have $\grf\shD_{X,D}\simeq \FA_{X,D}.$
\begin{definition}
A \emph{coherent} filtration $F_\bullet$ of a left $\shD_{X,D}$-module $\cM$ is an exhaustive increasing $\Z$-filtration bounded from below compatible with the order filtration of $\shD_{X,D}$ so that the associated graded module
$\grf\cM$ is \emph{coherent} over $\FA_{X,D}$. We then say $(\cM,F_\bullet)$ is coherent over $(\shD_{X,D},F_\bullet)$. We also say that $\cM$ is \emph{coherent} over $\shD_{X,D}$ if there exists a \emph{coherent} filtration $F_\bullet$. 
\end{definition}
Since $\pi$ is affine, $\widetilde{\gr}^F_\bullet\cM$ is a coherent $\sO_{T^*(X,D)}$-module. Analogous to the characteristic cycles of coherent $\shD_X$-modules, we define the \emph{characteristic cycle} of $\cM$  to be 
\[\SiS\cM=\sum m_p({\widetilde{\gr}^F_\bullet\cM})\cdot\bar p,\]
where $p$ goes over the generic points of irreducible components of the support of $\widetilde{\gr}^F_\bullet\cM$ and $m_p$ is the multiplicity of $\widetilde\gr^F_\bullet\cM$ at $p$. Similar to $\shD_X$-modules, the characteristic cycle is independent of the choices of coherent filtrations. 

Assume $x_1\partial_{x_1},\dots, x_k\partial_{x_k},\partial_{x_{k+1}}\dots,\partial_{x_n}$ are free generators of $\shTA_X(-\log D)$ locally with coordinates $(x_1,\dots,x_n)$ so that $D$ is defined by $x_1\cdots x_k=0$. Then we have a canonical section (independent of local coordinates)
$$\sum_i \frac{dx_i}{x_i}\otimes x_i\partial_{x_i}\in \Gamma(X, \Omega^1_X(\log D)\otimes\shTA_X(-\log D)).$$
The logarithmic \emph{de Rham complex} $\DR_D(\shD_{X,D})$ of $\shD_{X,D}$ is the following complex of right $\shD_{X,D}$-modules starting from the $-n$-term
\[\shD_{X,D}\stackrel{\nabla}{\rarr} \Omega^1_X(\log D)\otimes\shD_{X,D}\to \dots\to \Omega^n_X(\log D)\otimes \shD_{X,D}\]
where the first differential is 
\[P\mapsto \sum_i \frac{dx_i}{x_i}\otimes x_i\partial_{x_i}\cdot P\]
for a section $P$ of $\shD_{X,D}$. For a left $\shD_{X,D}$-module $\cM$, its \emph{de Rham complex} is
\be\label{eq:logdrd}
\DR_D(\cM)=\DR_D(\shD_{X,D})\otimes_{\shD_{X,D}}\cM.
\ee
It is also denoted as $\DR_{X,D}(\cM)$ if the ambient space $X$ needs to be emphasized.
In local coordinates, it is 
\begin{equation}\label{eq:drkos}
    \DR_D(\cM)\simeq\Kos(\cM;x_1\partial_{x_1},\dots, x_k\partial_{x_k},\partial_{x_{k+1}}\dots,\partial_{x_n}),
\end{equation}
the \emph{Koszul complex} of $\cM$ with actions of $x_1\partial_{x_1},\dots, x_k\partial_{x_k},\partial_{x_{k+1}}\dots,\partial_{x_n}$. 

In the case that $\cM$ has a filtration, $\DR_D(\cM)$ is filtered as follows:
\[F_p\DR_D(\cM)=[F_{p-n}\cM\to\Omega^1_X(\log D )\otimes F_{p-n+1}\cM\to\dots\to\Omega^n_X(\log D)\otimes F_p\cM].\]
We denote the associated graded complex by $\grf\DR_D(\cM)$, which is a complex of $\sAG_{X,D}$-modules.

Using Lie derivatives on $\omega_X$ and Lemma \ref{lm:pchrules}, the logarithmic canonical sheaf $\omega_X(D)=\Omega^n_X(\log D)$ is naturally a right $\shD_{X,D}$-module. We then have the evaluation map 
\[\textup{ev}\colon \Omega^n_X(\log D)\otimes \shD_{X,D}\to \omega_X(D)\]
given by 
\[\omega\otimes P\mapsto \omega\cdot P.\]
\begin{lemma}\label{lem:resomega}
The evaluation map \textup{ev} induces a locally free resolution of $\omega_X(D)$ as a right $\shD_{X,D}$-module
\[\DR_D(\shD_{X,D})\stackrel{\textup{ev}}{\rarr}\omega_X(D).\]
\end{lemma}
\begin{proof}
By construction, the associated graded complex $\grf\DR_D(\shD_{X,D})$ can be locally identified with the graded Koszul complex
\[\Kos(\sAG_{X,D};{\bf\xi})\]
where ${\bf\xi}=(\xi_1,\dots,\xi_k)$ are logarithmic symbols of $(x_1\partial_{x_1},\dots, x_k\partial_{x_k},\partial_{x_{k+1}}\dots,\partial_{x_n})$. It is obvious that $(\xi_1,\dots,\xi_k)$ is a regular sequence in $\sAG_{X,D}$. Since each $\xi_i$ is of degree 1, we know that $\grf\DR_D(\shD_{X,D})_k$ is acyclic for $k>0$. Therefore, the inclusion 
\[F_p\DR_D(\shD_{X,D})\rarr F_{p+1}\DR_D(\shD_{X,D})\]
is quasi-isomorphic for every $p>0$. Since $F_0\DR_D(\shD_{X,D})=\omega_X(D)$, we see that 
\[F_p\DR_D(\shD_{X,D})\stackrel{\textup{ev}}{\rarr}\omega_X(D)\]
is a quasi-ismorphism for every $p\ge 0$. Taking inductive limit as $p\to \infty$, as inductive limit functor is exact, we finish the proof. 
\end{proof}

Let us now introduce a special kind of $\shD_{X,D}$-modules, called lattices, following Ginsburg \cite{Gil}.  
\begin{definition}\label{def:lattice}
A (left) coherent $\shD_{X,D}$-module $\cM$ is called a \emph{lattice} of a coherent $\shD_U$-submodule $\cM_U$ if we have $\cM|_U\simeq \cM_U$ and
\[\cM\subseteq  j_*(\cM_U),\]
where $j\colon U=X\setminus D\hookrightarrow X$ is the open embedding. We also say that $\cM$ is a $\shD_{X,D}$-\emph{lattice} of $j_*(\cM_U)$ if $\cM$ is a \emph{lattice} of $\cM_U$.
\end{definition}

\begin{lemma}\label{lem:exlattice}
Assume $\cM_U$ is a coherent $\shD_U$-module. If $j_*\cM_U=\shD_{X}\cdot \cM_0$ for some coherent $\sO_X$-submodule $\cM_0$, then lattices of $\cM_U$ exist. In particular, if $\cM_U$ is regular holonomic, lattices of $\cM_U$ exist. 
\end{lemma}
\begin{proof}
We know that $\cM=\shD_{X,D}\cdot \cM_0(kD)$ is a lattice of $\cM_U$ for every $k\in \Z$. Hence, the first statement follows. If $\cM_U$ is regular holonomic, then $j_*(\cM_U)$ is also regular holonomic by definition (see for instance \cite[\S6]{HTT}). In particular, $j_*(\cM_U)$ is generated by a coherent $\sO_X$-module over $\shD_{X}$, from which the second statement follows.
\end{proof}

By Lemma \ref{lm:pchrules}, if $\cM$ is a lattice, then $\cM(E)=\cM\otimes\sO_X(E)$ is also a (left) $\shD_{X,D}$-module for every divisor $E$ supported on $D$ and hence they are all lattices of $\cM_U$ (one can easily check coherence of $\cM(E)$). Furthermore, if $\cM$ and $\cM'$ are two lattices of $\cM_U$, then for some $q>0$ we have 
\begin{equation}\label{eq:inclattice}
    \cM\subseteq \cM'(qD)\textup{ and } \cM'\subseteq \cM(qD).
\end{equation}

\subsection{Characteristic cycles}

The Lie bracket of $\shTA_X$ induces a Lie bracket of $\shTA_{X}(-\log D)$. Hence $T^*(X,D)$ is a logarithmic symplectic manifold with a symplectic form having log poles along $D$. An ideal sheaf $\mathcal I$ of $\sO_{T^*(X,D)}$ is called \emph{involutive} if 
\[\{\cI,\cI\}\subseteq \cI\]
where $\{\bullet,\bullet\}$ is the \emph{Poisson bracket} induced by the logarithmic symplectic form. 
A subvariety of $T^*(X,D)$ is called \emph{involutive} if its ideal sheaf is so. By \emph{Gabber's involutivity theorem} (see for instance \cite[A:III.3]{Bj}), we obtain:
\begin{theorem}
If $\cM$ is a coherent $\shD_{X,D}$-module, then the support of $\SiS\cM$ is involutive. 
\end{theorem}
It is worth mentioning that the dimension of logarithmic characteristic cycles does not satisfy the Bernstein inequality. For instance, considering the $\shD_{\bbC, 0}$-module $\bbC=\bbC[x]/x\bbC[x]$, its characteristic cycle is the point $\{0\}\subset\bbC\simeq T^*_\bbC(\bbC,0)\subset T^*(\bbC,0)$. The defect of Bernstein inequality can be fixed by considering log dimensions; see \cite[\S 3.3]{KT}.

\section{Direct images of $\shD_{X,D}$-modules}\label{sec:directimage}
We discuss direct image functors for $\shD_{X,D}$-modules in this section. Suppose $f\colon (X, D)\to (Y,E)$ is a morphism of log smooth pairs, that is, $(X,D)$ and $(Y,E)$ are two log smooth pairs and $f$ is a morphism between $X$ and $Y$ so that $f^{-1}E\subseteq D$. We also say that $f$ is a log morphism in this case. 

Analogous to $\shD$-modules, we define the transfer module $\shD_f$ of the log morphism $f$ by 
\[\shD_f=\omega_f\otimes f^{*}\shD_{Y,E}\]
where $\omega_f=\omega_{X}(D)\otimes f^*\omega_Y(E)^{-1}$, the relative logarithmic canonical sheaf. Clearly, $\shD_f$ is a right $\shD_{X,D}$ and left $f^{-1}\shD_{Y,E}$ bi-module with a filtration $F_\bullet$ induced from the order filtration of $\shD_{Y,E}$. The direct image functor between bounded derived categories of left logarithmic $\shD$-modules
\[f_+\colon D^b(\shD_{X,D})\to D^b(\shD_{Y,E})\]
is given by 
\[f_+(\bullet)=Rf_*(\shD_{f} \stackrel{\bf L}{\otimes}_{\shD_{X,D}} \bullet).\]
We set $D^b_\coh(\shD_{X,D})$ to be the bounded derived categories of left logarithmic $\shD$-modules with coherent cohomology sheaves. 

\begin{theorem}\label{thm:coh}
For a log morphism $f\colon (X,D)\to (Y,E)$, if the morphism $f\colon X\to Y$ is proper, then $f_+$ preserves coherence, that is 
\[f_+\colon D^b_\coh(\shD_{X,D})\to D^b_\coh(\shD_{Y,E}).\]
\end{theorem}
\begin{proof}
If $\cM=\shD_{X,D}\otimes_{\sO_X}\sL$ for some coherent $\sO_X$-module $\sL$ ($\cM$ is called an induced $\shD_{X,D}$-module), then 
\[f_+\cM\in D^b_\coh(\shD_{Y,E})\]
by projection formula and the standard fact in algebraic geometry that the direct images of $\sO$-modules under proper morphisms are coherent. In general, for an arbitrary $\cM^\bullet\in D^b_\coh(\shD_{X,D})$, using the arguments in the proof of \cite[Theorem 1.5.8]{Bj}, $\cM^\bullet$ is quasi-isomorphic to a complex of induced $\shD_{X,D}$-modules. Then we also have 
\[f_+\cM^\bullet\in D^b_\coh(\shD_{Y,E})\]
by induction on the length of the complex of induced $\shD_{X,D}$-modules. 
\end{proof}

\subsection{Direct images and logarithmic Lagrangian correspondence}
In this subsection, we discuss the direct image functor $f_+$ in the filtered case, which is the logarithmic generalization of Laumon's constructions for $\shD$-modules \cite{Laumon,Laumon1}.

For a log morphism $f\colon (X,D)\to (Y,E)$, we consider the following diagram 
\[\begin{tikzcd}
T^*(X,D)\arrow[dr,"\pi_X"] &  X\times_YT^*(Y,E)\arrow[l,"p_1"] \arrow[r, "p_2"] \arrow[d,"\pi_f"]
    & T^*(Y,E) \arrow[d, "\pi_Y"] \\
 &  X \arrow[r, "f"]
&Y, \end{tikzcd}\]
where the square is Cartesian and $p_1$ is induced by the morphism of logarithmic tangent sheaves $\shTA_{X}(-\log D)\to f^*\shTA_Y(-\log E)$. 

Denoted the bounded derived categories of the graded $\sAG_{X,D}$-modules and $\sO_{T^*(X,D)}$-modules  by $D^b(\sAG_{X,D})$ and $ D^b(\sO_{T^*(X,D)})$ respectively.
We define a functor 
\[f_\sharp\colon D^b(\sO_{T^*(X,D)})\rarr D^b(\sO_{T^*(Y,E)})\]
associated to $f$ by 
\[f_\sharp(\bullet)=Rp_{2*}(Lp_1^*(\bullet)\otimes p_1^*\omega_f).\]
When $f$ is proper, $f_\sharp$ preserves coherence. 

Now, we consider the additive category $\MF(\shD_{X,D})$ whose objects are filtered (left) $\shD_{X,D}$-modules and morphisms are $\shD_{X,D}$-linear morphisms compatible with the order filtration of $\shD_{X,D}$. A sequence of $\MF(\shD_{X,D})$
\[(\cM_1,F_\bullet)\to (\cM_2,F_\bullet)\to (\cM_3,F_\bullet)\]
is called strictly exact if 
\[0\to \grf\cM_1\to \grf\cM_2\to \grf\cM_3\to0\]
is a short exact sequence of $\sAG_{X,D}$-modules. Then $\MF(\shD_{X,D})$ together with the class of all the strictly exact sequences defines an exact categories in the sense of Quillen. Denote the bounded Derived category of $\MF(\shD_{X,D})$ by $\DF^b(\shD_{X,D})$. Then we have an embedding of categories 
\[\MF(\shD_{X,D})\lhook\joinrel\xrightarrow{R(\bullet)} \Mod(R(\shD_{X,D},F_\bullet)),\]
and an induced embedding of (bounded) Derived categories 
\[\DF^b(\shD_{X,D})\hookrightarrow D^b(R(\shD_{X,D},F_\bullet)),\]
where $R(\shD_{X,D},F_\bullet)=\bigoplus_iF_i\shD_{X,D}$, the Rees ring of $(\shD_{X,D},F_\bullet)$, $\Mod(R(\shD_{X,D},F_\bullet))$ the abelian categories of $R(\shD_{X,D},F_\bullet)$-modules and  $D^b(R(\shD_{X,D},F_\bullet))$ the bounded derived categories. A complex $\cM^\bullet\in \DF^b(\shD_{X,D})$ is coherent if its image in $D^b(R(\shD_{X,D},F_\bullet))$ has coherent cohomologies. Denoted by $\DF^b_{\coh}(\shD_{X,D})$ the subcategory of coherent objects. 

The direct image functor can be naturally generalized to the Rees modules case as following:
\[f_+\colon D^b(R(\shD_{X,D},F_\bullet))\to D^b(R(\shD_{Y,E},F_\bullet))\]
\[f_+(\bullet)=Rf_*(R(\shD_f,F_\bullet)) \stackrel{\bf L}{\otimes}_{R(\shD_{X,D},F_\bullet)} \bullet,\]
where the filtration of $\shD_f$ is induced from that of $\shD_{Y,E}$.
By abuse of notations, we denote the direct images functors on different derived categories all by $f_+$. Similar to the unfiltered case, $f_+$ preserves coherence when $f$ is proper (see the proof of Theorem \ref{thm:coh}), that is, we have 
\[f_+\colon D^b_\coh(R(\shD_{X,D},F_\bullet))\to D^b_\coh(R(\shD_{Y,E},F_\bullet)).\]
However, $f_+$ does not preserve strictness, that is, it might not be true that $f_+$ maps $\DF^b(\shD_{X,D})$ into $\DF^b(\shD_{Y,E})$ even when $f$ is proper.

We then define an intermediate functor
\[f^\bullet_+\colon D^b(\sAG_{X, D})\to D^b(\sAG_{Y,E})\]
by 
\[f^\bullet_+(\bullet)=Rf_*(\grf\shD_f \stackrel{\bf L}{\otimes}_{\sAG_{X,D}}\bullet).\]
 
Since $\pi_{-}$ is affine, where $- = X$ or $Y$, we have the functor 
\[\sim_-\colon  D^b(\sAG_{-, -})\to D^b(\sO_{T^*(-,-)}).\]
 
Then we obtain the following proposition:
\begin{prop}\label{prop:Laumf}
Suppose $f\colon (X, D)\to (Y,E)$ is a proper morphism of log smooth pairs. Then we have the following commutative diagram 
\[\begin{tikzcd}
D^b_\coh(\sAG_{X, D})\arrow[r,"f^\bullet_+"]\arrow[d,"\sim_X"]& D^b_\coh(\sAG_{Y,E})\arrow[d,"\sim_Y"]\\
 D^b_\coh(\sO_{T^*(X,D)})\arrow[r,"f_\sharp"] & D^b_\coh(\sO_{T^*(Y,E)}).
\end{tikzcd}\]
\end{prop}
We use $K_\coh(\bullet)$ (resp. $\KF_\coh(\bullet)$) to denote the Grothendieck group of the triangulated category $D^b_\coh(\bullet)$ (resp. $\DF^b_\coh(\bullet)$). We then have a Laumon-type formula for log $\shD$-modules, roughly speaking,
\[\gr\circ f_+\simeq f^\bullet_+\circ\gr.\]
To be more precise: 
\begin{theorem}\label{thm:Laumf}
Suppose $f\colon (X, D)\to (Y,E)$ is a proper morphism of log smooth pairs. Then we have the following commutative diagram 
\[\begin{tikzcd}\KF_\coh(\shD_{X,D})\arrow[r,"f_+"]\arrow[d,"{[\gr]}"]& \KF_\coh(\shD_{Y,E})\arrow[d,"{[\gr]}"]\\
K_\coh(\sAG_{X, D})\arrow[r,"f^\bullet_+"]& K_\coh(\sAG_{Y,E}).
\end{tikzcd}\]
\end{theorem}
\begin{proof}
For a filtered complex $(\cM^\bullet, F_\bullet)\in \DF^b(\shD_{X,D})$,  $f_+(\cM^\bullet, F_\bullet)$ is a filtered complex. The associated spectral sequence converges and for every $p\in \Z$
$$E_r^p=\bigoplus_q E_r^{p,q}\simeq \grf\sH^pf_+(\cM^\bullet)$$
with the induced filtration on $\sH^pf_+(\cM^\bullet)$ when $r\gg0$. 

Take the natural $t$-structure on $\DF^b(\shD_{Y,E})$ (use the usual truncation functor), and denote the heart by $\sC$ and the associated $n$-th cohomology functor by $H^n_t$. Then the objects of $\sC$ are 2-complexes of injective morphisms (not necessarily strict) of filtered $\shD_{Y,E}$-modules
$$\eta\colon(\cM^{-1}, F_\bullet)\stackrel{d^{-1}}{\rarr} (\cM^0,F_\bullet).$$ 
Define a functor from $\sC$ to $\MF(\shD_{Y,E})$ by 
$$H^0(\eta)=\Coker(\Coim d^{-1}\to \Ker d^0).$$ 
It induces additive functors $H^n=H^0\circ H^n_t$ on $DF^b(\shD_{Y,E})$; they are not necessarily cohomological. Then one can check 
\begin{equation}\label{eq:comp1}
    [\grf H^0(\eta)]=[\gr\eta]
\end{equation}
in $K_\coh(\sAG_{Y,E})$ (see for instance \cite[Lemme 3.5.13 (iii)]{Laumon}).

On the other hand, consider the spectral sequence associated to $\eta$. By convergency of the spectral sequence we have that for $r\gg0$ 
\begin{equation}\label{eq:comp2}
    E^0_r=\grf H^0(\eta).
\end{equation}
Combining \eqref{eq:comp1} and \eqref{eq:comp2}, the proof is accomplished, thanks to the fact that taking limit of spectral sequences and truncation operations commute.  
\end{proof}
By combining Proposition \ref{prop:Laumf} and Theorem \ref{thm:Laumf}, we then immediately have:
\begin{coro}\label{cor:imageK}
Suppose $f\colon (X, D)\to (Y,E)$ is a proper morphism of log smooth pairs. Then for $(\cM^\bullet, F_\bullet)\in \DF^b_\coh(\shD_{X,D})$,
\[f_\sharp[\widetilde{\gr}^F_\bullet \cM^\bullet]=\sum_i[(-1)^i\widetilde{\gr}^F_\bullet (\sH^if_+(\cM^\bullet))]\]
in $K_\coh(\sO_{T^*(Y,E)})$. 
\end{coro}




\section{Logarithmic comparisons for lattices}\label{sec:bfs}

\subsection{Sabbah's multi-filtrations and generalized Bernstein-Sato polynomials}\label{sec:sabreview}

In this subsection, we review the result from \cite{Sab} about multi-filtrations of coherent $\shD$-modules. 



Let $Y$ be a smooth algebraic variety (or more generally, a complex manifold) and $H \subset Y$ a smooth hypersurface. The \emph{Kashiwara-Malgrange filtration} $\{V_i\shD_Y\}_{i\in \Z}$ on $\shD_Y$ along $H$ is an increasing filtration defined by: 
\[V_i\shD_{Y}=\{P\in \shD_Y|P\cdot I^j\subseteq I^{j-i} \textup{ for }\forall j\in \Z\},\]
where $I$ is the ideal sheaf of $H$ and $I^j=\sO_Y$ for $j\le0$. 


For simple normal crossing (SNC) divisors, there is a notion of multi-filtration on $\shD_X$ and local good coherent $\shD_X$-modules. Working locally, we may assume that $X=\Delta^n$ is the $n$-dimensional polydisc with coordinates $(x_1,\dots,x_n)$, and smooth divisor $D_l=(x_l=0)$ for $l=1,2,\dots,k$ and $k\le n$. 
We write $^j V_\bullet\shD_X$ the Kashiwara-Malgrange filtration of $\shD_X$ along $D_j$, and then set for $\bs=(s_1,\dots,s_k)\in \Z^k$
\[V_\bs\shD_X=\bigcap_{j=1}^k {}^j V_{s_i}\shD_X.\]
We then obtain the (multi-indexed) $k$-filtration $\{V_\bullet\shD_X\}_{\Z^k}$. We write the associated Rees ring by
\[R_V(\shD_X)=\bigoplus_{\bs\in \Z^k}V_\bs\shD_X \cdot u^\bs; \quad u^\bs = u_1^{s_1} \cdots u_k^{s_k}\]
one can check that $R_V(\shD_X)$ is a (graded) coherent and Noetherian sheaf of rings. 

We fix a $\shD_X$-module $\widetilde{\cM}$ and consider a $k$-filtration $\{U_\bullet \wt \cM\}_{\Z^k}$ compatible with $\{V_\bullet\shD_X\}_{\Z^k}$. The filtration $\{U_\bullet\widetilde\cM\}_{\Z^k}$ is called good if the Rees module 
\[ R_U(\widetilde{\cM}) = \bigoplus_{\bs\in \Z^k}U_\bs \cM \cdot u^\bs;  \]
is coherent over $R_V(\shD_X)$. It is worth mentioning that if $\{U_\bullet\widetilde\cM\}_{\Z^k}$ is good, then $\widetilde{\cM}$ is coherent and conversely, if $\widetilde\cM$ is coherent, then good $k$-filtrations exist locally. For notation simplicity, we denote $U_\bs \wt \cM$ as $U_\bs$.

As remarked in \cite{Sab},  the subtlety with multi-filtration on module (as in contrast with $\shD_X$) is that in general
\[ U_{\bs} \subsetneq {}^1 U_{s_1} \cap \cdots \cap {}^k U_{s_k},  \text{ where } \quad ^j U_{s'_j} = \bigcup_{\bs \in \Z^k s_j = s'_j} U_\bs. \]
For example, for $k=2$, we do not have $U_{(0,0)} = U_{(0,1)} \cap U_{(1,0)}$ in general.  

It is precisely for this reason, that Sabbah introduces a refined filtration with respect to a cone $\Gamma$. To give a precise definition, we introduce the following notation
\[ N = \Z^k, \quad N^+ = (\Z_{\geq 0})^n, \quad N_\Q = N \otimes \Q, \quad N_\Q^+ = (\Q_{\geq 0})^k \] 
Let $M$ be the dual lattice of $N$, and define $M_\Q$, $M_\Q^+$ accordingly. Let $\Gamma \subset N^+_\Q$ be a unimodular simplicial cone contained in the positive 
quadrant of $N_\Q$, such that the primitive generators of the rays in $\Gamma$ forms part of a $\Z$-basis of the $N$. We use $\sL(\Gamma)$ to denote these primitive generators. We also define the dual cone \[ \check{\Gamma} = \{m \in M | \langle m, v \rangle \geq 0, \forall v \in \Gamma\}, \]
and the annihilator of $\Gamma$ in $M$ by $\Gamma^\perp$. We denote a partial ordering on $M$ induced by $\Gamma$ by 
\[ s \leq_\Gamma s' \LRA  s' - s \in \check{\Gamma}, \]
and we say
\[ s <_\Gamma s' \LRA  s \leq_\Gamma s' \text{ but not } s' \leq_\Gamma s. \]

\begin{definition}
Let $U_\bu$ be a good $k$-filtration of $\wt \cM$ with respect to $V_\bu \shD_X$. Let $\Gamma$ be a $k'$-dimensional unimodular simplicial cone contained in the positive quadrant. For any element $s \in M$, we define a new $k$-filtration by \[ {}^\Gamma U_s = \sum_{s' \leq_\Gamma s} U_{s'}.\]
\end{definition}
Note that if $k' < k$, then ${}^\Gamma U_s$ only depends on the image of $s$ in $M / \Gamma^\perp$. Hence, in the special case that $L$ is a one-dimensional cone (or abusing notation, a primitive generator of this cone), we write ${}^L U_\lambda$ for the $\Z$-indexed filtration, where $\lambda \in \Z \cong M / L^\perp$. 

Let $R_\Gamma(\wt \cM) = \bigoplus_{s \in M} {}^\Gamma U_s u^s$ denote the Rees module for  ${}^\Gamma U_\bu$. There is a natural $\C[\check{\Gamma} \cap M]$ action on 
$R_\Gamma(\wt \cM)$, namely for $s \in \check{\Gamma}$. We recall the following property.
\begin{lemma}\cite[Lemme 2.2.2]{Sab}
If $R_\Gamma(\wt \cM)$ is a flat $\C[\check{\Gamma} \cap M]$-module, then for all $s \in M$, we have
\[ {}^\Gamma U_s = \bigcap_{L \in \sL(\Gamma)} {}^L U_{L(s)}. \]
\end{lemma}
If ${}^\Gamma U_\bu$ satisfies the flatness condition, we call such cone $\Gamma$ {\em adapted} to $U_\bu$. 
In general, the standard cone $M^+_\Q$ is not adapted to $U_\bu$. However Sabbah shows one may subdivide the standard cone to get an {\em adapted fan} $\Sigma$, that is every cone in $\Sigma$ is adapted. 

Let $\sL(\Sigma)$ denote the set of rays in $\Sigma$. We also define a new $k$-filtration by
\[\overline{U}_\bs=\bigcap_{L\in \sL(\Sigma)}\lup LU_{L(\bs)}\]
for every $\bs\in \Z^k$.  This is called the saturation of $U_\bullet\widetilde\cM$ with respect to $\Gamma$. The saturation filtration is also good provided that $U_\bullet\widetilde\cM$ is good (see \cite[Proposition-D\'efinition 2.2.3]{Sab}). By definition, we have for every $\bs\in \Z^k$
\[U_{\bs}\subseteq \overline{U}_\bs;\]
the $k$-filtration $U_\bullet\widetilde{\cM}$ is called saturated if the above inclusion is an equality. 

The following beautiful theorem of Sabbah \cite[Th\'eor\`em de Bernstein]{Sab} is a natural generalization of the existence of the Bernstein-Sato polynomials for regular functions. 
\begin{theorem}[Sabbah]\label{thm:Sabbf}
Suppose that $U_\bullet\widetilde{\cM}$ is a good $k$-filtration. Then for every primitive vector $L \in N^+$ in the first quadrant, there exists a polynomial of one-variable $b_L(s)\in \bbC[s]$, such that for every $\lambda\in \Z$ \footnote{We suspect there is a typo in \cite[3.1.1. Th\'eor\`em de Bernstein]{Sab}, missing a $+\lambda)$ in the formula $b_L( L(...) {}^L U_\lambda $. Also, we use operator $x_i \partial_{x_i}$ instead of $\partial_{x_i} x_i$, resulting a possible difference in coefficients in $b_L$. }
\[b_L(L(x_1\partial_{x_1},\dots,x_k\partial_{x_k})+\lambda)\lup LU_{\lambda}\subseteq \lup LU_{\lambda-1}.\]

\end{theorem}

\subsection{Bernstein-Sato polynomials for lattices}
In this subsection, we prove the existence of Bernstein-Sato polynomials for lattices using Sabbah's multi-filtrations. 

We continue to assume that $X=\Delta^n$ is the $n$-dimensional polydisk with coordinates $(x_1,x_2,\dots,x_n)$, and smooth divisor $D_l=(x_l=0)$ for $l=1,2,\dots,k$ and $k\le n$. We set $D=\sum_{l=1}^kD_l$ and let $\cM$ be a $\shD_{X,D}$ lattice of some regular holonomic $\shD_U$-module $\cM_U$, where $j\colon U=X\setminus D\hookrightarrow X$.

We obtain a good $k$-filtration $U_\bullet\widetilde{\cM}$ of $\widetilde\cM=\shD_X\cdot \cM$ associated to $\cM$ by requiring:
\be\label{eq:kflattice}
R_U(\widetilde\cM)=R_V(\shD_X)\cdot \bigoplus_{\bs\le 0}\cM(\sum_{l=1}^ks_lD_l)\subseteq \bigoplus_{\bs\in \Z^k}\widetilde\cM;
\ee
it is good because $\bigoplus_{\bs\le 0}\cM(D\cdot \bs)$ is coherent over $\bigoplus_{\bs\le 0}V_\bs\shD_{X,D}$.
In other words, we have 
\[ U_\bs \wt \cM = V_\bs \shD_X \cdot \cM.\]
In particularly, we have for $\bs\le 0$
\[U_{\bs}\widetilde{\cM}=\cM(\sum_{l=1}^ks_lD_l).\]

The following theorem is a generalization of \cite[Proposition 1.2]{Sab2}. Although the statement there only concerns with log D-modules coming from the graph embedding, but the proof carries through exactly.
\begin{theorem}\label{thm:bflattice}
Assume that $\cM$ is a $\shD_{X,D}$ lattice of some regular holonomic $\shD_U$-module $\cM_U$. Then locally around a point $x\in D$ there exists $b_\cM(s_1, \cdots, s_k) \in \C[\bs]$, such that
\[ b_\cM({\bf x\partial}) \cM \In \cM (-D). \] 
Moreover, $b_\cM$ can be factorized as product of linear functions of the form $c + \sum_{i=1}^k \alpha_i s_i $ where $\alpha_i \in \Q_{\geq 0}$. 
\end{theorem}
\begin{proof}
Write $\wt \cM = \shD_Y\cdot\cM$. 
Let $U_\bu \wt \cM = V_\bu \shD_X \cdot \cM$ be the filtration associated to $\cM$ as in \eqref{eq:kflattice}. 
Let $\Sigma$ be a fan in $N^+_\Q$ adapted to $U_\bu \wt \cM$. Let $L_1, \cdots, L_m$ be the collection of primitive vectors in the rays of $\Sigma$. Let $b_{L_i}(s)$ be the corresponding Bernstein-Sato polynomial for ${}^{L_i} U_\bu$ as in Theorem \ref{thm:Sabbf}.

Then one can define a polynomial $b(s_1, \cdots, s_n)$, given by
\[ b(\bs) = \prod_{i=1}^m \prod_{j=0}^{L_i(\vec 1)-1} b_{L_i}(L_i(\bs) -j) \]
where $\vec 1 = (1,\cdots, 1) \in \Z^k$, 
such that, 
\[ b(x_1 \pa_{x_1}, \cdots, x_k \pa_{x_k}) \overline U_{\vec 0} \In \overline U_{ - \vec 1},\]
where $\overline U_\bullet$ is the saturation of $U_\bullet$.
Hence, for all $\vec a \in \Z^k$, we have
\[ b(x_1 \pa_{x_1}+a_1, \cdots, x_k \pa_{x_k}+a_k) \overline U_{\vec a} \In \overline U_{ \vec a - \vec 1}.\]
Hence for any $\vec a$, and any $\Z \ni \lambda \geq 1$, we have a $b$-function 
\[ b_{\vec a, \lambda} (\{x_i \pa_i\}) \overline U_{\vec a} \In \overline U_{ \vec a - \lambda \cdot \vec 1}. \]

Since $\overline U_\bu$ is a good filtration of $\wt \cM$, in particular, there is a $N_1 \gg 0$, such that for all $\vec a \in (\Z_{\geq 0})^k$, we have 
\[ V_{-\vec a} \shD_X \cdot \overline U_{-N_1 \cdot \vec 1} = \overline U_{-N_1 \cdot \vec 1 - \vec a}. \]
And we also have
 \[ V_{-\vec a} \shD_X \cdot \overline U_{-N_1 \cdot \vec 1} = x_1^{a_1} \cdots x_k^{a_k} \overline U_{-N_1 \cdot \vec 1}. \]
 Since $\overline U_{-N_1 \cdot \vec 1}$ is a lattice for $\wt \cM_U$, we can find $N_2 \gg 0$, such that
 \[ \overline U_{-N_1 \cdot \vec 1} \In (x_1 \cdots x_k)^{-N_2}  \cM . \]
Hence we get
\[ \overline U_{-(N_1+N_2+1) \cdot \vec 1} \In \cM(-D) \In \cM = U_{\vec 0} \In  \overline U_{\vec 0}. \]
Thus the following construction will work
\[ b_\cM(s_1, \cdots, s_k) =  b_{\vec 0, (N_1+N_2+1)} (s_1, \cdots, s_k). \]

\end{proof}

\subsection{Proof of Theorem \ref{thm:logcom}}

Since the required statement is local, we can assume 
$$X=N\times T$$
with coordinates $(x_1,\dots,x_k,x_{k+1},\dots,x_n)$ centered at a point $x\in D$,
so that $(x_1,\dots,x_k)$ are coordinates of $N$ and $(x_{k+1},\dots,x_n)$ are coordinates of $T$ and 
$$D=\bigcup_{i=1}^k (x_i=0).$$ 

Since for every $q\in \Z$, $\cM(qD)$ is a lattice, we have a short exact sequence of $\shD_{X,D}$-modules
\begin{equation}\label{eq:sesp}
    0\to \cM((q-1)D)\to \cM(qD)\to \frac{\cM(pD)}{\cM((q-1)D)}\to 0.
\end{equation}
We first prove that 
\begin{equation}
    \DR_D(\frac{\cM(qD)}{\cM((q-1)D)}) \textup{ are acyclic for all } |p|\gg0.
\end{equation}
To this purpose, we apply Theorem \ref{thm:bflattice} to the lattice $\cM$ and obtain the linear forms $L_1,\dots,L_m$. Meanwhile, the $\shD_{X,D}$-module structure of $\cM$ makes the stalks
\[(\frac{\cM(qD)}{\cM((q-1)D)})_x\] 
$\bbC[s_1, \cdots, s_k]\simeq\bbC[{\bf x\partial}]$-modules for all $q\in \Z$. Thanks to Theorem \ref{thm:bflattice} again, the support of $\bbC[\bs]$-module 
\[(\frac{\cM}{\cM(-D)})_x\]
is contained in the zero locus of $b_{L_i}(L_i(\bs))$ for $i=1, \cdots, m$. Note the zero locus of $b_{L_i}(L_i(\bs))$ is a union of parallel hyperplanes in $M_\Q$ with co-vector $L_i$. Also since $L_i$ is in the first quadrant of $N$, we have $\la L_i (1, \cdots, 1) \ra > 0$, that is the line $\Q \cdot (1, \cdots, 1)$ pass through the zero-locus of $b_{L_i}(L_i(\bs))$ only finitely many times.  

Since 
\[\frac{\cM(qD)}{\cM((q-1)D)}\simeq \prod_{l=1}^k x_l^{-q}\cdot (\frac{\cM}{\cM(-D)}),\]
$(\frac{\cM(qD)}{\cM((q-1)D)})_x$ is supported on the  zero locus of $b_{L_i}(L_i(s_1 + q, \cdots, s_k+q))$, which is the zero locus of $b_{L_i}(L_i(\bs))$ shifted by $(-q, \cdots, -q)$. Hence for  $|q|$ large enough, $(0, \cdots, 0)$ is not contained in the support of $(\frac{\cM(qD)}{\cM((q-1)D)})_x$ as a $\C[\bs]$-module.
The Koszul complex 
\[\Kos((\frac{\cM(qD)}{\cM((q-1)D)})_x;{\bf x\partial})\]
is identified with the complex 
\[(\frac{\cM(qD)}{\cM((q-1)D)})_x\otimes_{\bbC[\bs]}\Kos(\bbC[\bs];\bs).\]
As the Koszul complex $\Kos(\bbC[\bs];\bs)$ is supported exactly on $\{0\}$, the complexes 
\[\Kos((\frac{\cM(qD)}{\cM((q-1)D)})_x;{\bf x\partial})\]
are acyclic for all $|q|\gg0$. Moreover, by \eqref{eq:drkos}, we see that
\[\DR_D(\frac{\cM(qD)}{\cM((q-1)D)})_x\simeq\DR_T(\Kos((\frac{\cM(qD)}{\cM((q-1)D)})_x;{\bf x\partial})),\]
where $\DR_T$ means the \emph{de Rham} functor is applied on the ambient space $T$ instead of $X$. 
Hence we conclude that they are both acyclic for every $|q|\gg 1$. 

Considering the short exact sequence \eqref{eq:sesp}, since \emph{de Rham} functor is exact, we obtain that 
\[\DR_D(\cM((q-1)D))\rarr \DR_D(\cM(qD))\]
is a quasi-isomorphism for every $|q|\gg 1$. We then take the inductive limit as $q\to \infty$. Since the inductive limit functor is exact, the natural morphism 
\[\DR_D(\cM(qD))\to \lim_{q\to \infty}DR_D(\cM(qD))=\DR(j_*\cM_U).\]
is a quasi-isomorphism $q\gg 1$. Since $\cM_U$ is regular holonomic, we know  $\DR$ and algebraic localizations commute (see \cite[Chapter V.4]{Bj}), and hence we have
\[\DR(j_*\cM_U)\simeq Rj_*\DR(\cM_U).\]
Consequently, we obtain the quasi-isomorphism \eqref{eq:logrh}.
\qed


\subsection{An example: Deligne lattices}\label{sec:dl}
Suppose that $(X,D)$ is a smooth log pair with $\dim X=n$. We fix a $\bbC$-local system $L$ on 
$U^\an=(X\setminus D)^\an$
and set 
$$\cV^{\an}=L\otimes \sO^\an_U$$ 
the flat holomorphic vector bundle. Denoted by $\overline{\cV}$ the Deligne lattice with eigenvalues of residues along $D$ having real parts in $(-1,0]$. It is well-known that the construction of Deligne lattices is in analytic nature; see \cite{Del} and \cite[\S 5]{HTT}. More precisely, $\overline{\cV}$ is a locally free $\sO^{\an}_X$-module of finite rank. Since $X$ is algebraic, $\overline\cV$ is also algebraic by GAGA (by adding more boundary divisors, $X$ can be assumed to be complete). We take $\cV=\overline\cV|_U$, the algebraic $\shD_U$-module of the local system $L$. Then $\overline \cV$ is a $\shD_{X,D}$ lattice of $j_{*}\cV$, where $j\colon U\hookrightarrow X$. 

We write $D=\sum_{l=1}^kD_l$. For a subset $I\subseteq \{1,2,\dots,k\}$, we also write 
\[D^{I}=\sum_{l\in I}D_l \textup{ and }D^{\bar I}=\sum_{l\notin I}D_l\]
and 
$$j^I_1: U\to X\setminus D^I \textup{ and }j^I_2: X\setminus D^I\to X.$$

We now calculate the $b$-functions for $\overline\cV$ locally.
Assume that the local system $L$ is of rank $m$ and let $\{c_1,c_2,\dots,c_m\}$ be a (linear independent) set of multi-valued sections of $L$ around $x\in D$. For simplicity we write
\[V=\textup{Span}_\bbC\{c_1,c_2,\dots,c_m\}.\]
Then $\overline \cV_x$ is trivialized (analytically) as $\sO_x$-module by 
\be\label{eq:dltriv}
\{e_i=\exp({\sum_{l=1}^k\Gamma_l\log x_l)\cdot c_i\}_{i=1}^m)},
\ee
where $\Gamma_l\in \mathfrak{gl}(V,\bbC)$ satisfying that the eigenvalues of $\Gamma_i\in (-1,0]$ for every $l$; in other words, $\Gamma_l$ is one branch of the logarithm of the monodromy of $L$ along the divisor $(x_l=0)$. One sees easily that $x_l\partial_{x_l}$ operates naturally on $e_i$. In this case, the $b$-function is: for every $l$ and for every pair of $k_1,k_2\in \Z$
\be\label{eq:bflattices}
(x_l\partial_{x_l}+\lambda_l)^{n_l}\cdot\overline{\cV}(k_2D^{\bar I}-k_1D^{I})_x\subseteq\overline{\cV}(k_2D^{\bar I}-k_1D^{I}-D_l)_x
\ee
with the real part Re$(\lambda_l)\in (-k_2-1,-k_2]$ if $l\not\in I$ and Re$(\lambda_l)\in (k_1-1,k_1]$ if $l\in I$.
\begin{proof}[Proof of Theorem \ref{thm:DLlc}:]
When $x\in D^{\bar I}\setminus D^I$, Theorem \ref{thm:logcom} and \eqref{eq:bflattices} imply for $k_2>0$
\[\DR_D(\overline{\cV}(k_2D^{\bar I}-k_1D^{I}))_x\simeq (Rj^I_{1*}L[n])_x.\]
Therefore, we still need to prove the case when $x\in D^I$. We now prove this case. Using the analytic trivialization of $\overline{\cV}(k_2D^{\bar I}-k_1D^{I})_x$ induced by \eqref{eq:dltriv}, one easily checks that 
\[x_l\partial_{x_l}: \overline{\cV}(k_2D^{\bar I}-k_1D^{I})_x\rarr\overline{\cV}(k_2D^{\bar I}-k_1D^{I})_x\]
is an isomorphism for $l\in I$ and for $k_1,k_2>0$. Hence, we have that
\[\DR_D(\overline{\cV}(k_2D^{\bar I}-k_1D^{I}))_x\]
is acyclic for $x\in D^I$. Since $j^I_{2!}$ is $0$-extension along $D^I$, the proof is now accomplished.
\end{proof}

\section{Application to $\shD_Y[\bs](\bh^{\bs+\vb}\cdot \cM_0)$}\label{sec:bbloc}
Suppose that $\bh=(h_1,\dots,h_k)$ is a $k$-tuple of regular functions on a smooth variety $Y$ of dimension $m$. Let $U=Y\setminus \prod_lh_l=0$ and $j:U\hookrightarrow Y$ be the open embedding. Let $\cM_U$ be a regular holonomic $\shD_{U}$-module. Since $\widetilde\cM=j_*\cM_U$ is also regular holonomic, we can assume that $\wt{\cM}$ is generated over $\shD_Y$ by 
 some $\sO_Y$-coherent submodule $\cM_0$. With respect to $\cM_0$, we consider the $\shD_Y[\bs]$-module generated by $\bh^{\bs+\vb}$ for $\vb=(v_1,\dots,v_l)\in \Z^k$:
\[\cM^\vb_\bh=\shD_Y[\bs](\bh^{\bs+\vb}\cdot \cM_0)\subseteq j_*(\bh^{\bs}\cdot\cM_U[\bs])=\bh^{\bs}\cdot\widetilde\cM[\bs]\]
where $\bh^{\bs+\vb}=\prod_lh_l^{s_l+v_l}$ and $\bs=(s_1,\dots,s_k)$ are independent variables.

On the other hand, consider the graph embedding of $\bh$:
\[\eta^\bh\colon Y\hookrightarrow X=Y\times \bbC^k\]
given by 
\[y\mapsto (y, h_1(y),\dots,h_k(y)) \textup{ for } y\in Y.\]
We write the coordinates of $\bbC^k$ by $(t_1,\dots,t_k)$. By identifying 
$  s_l \textup{ with } -t_l\partial_{t_l} $
we have 
\[j_*(\bh^{\bs}\cdot\cM_U[\bs])\simeq \eta^\bh_{+}\widetilde\cM\]
as $\shD_X$-modules, where $\eta^\bh_{+}$ denotes the $\shD$-module pushforward of $\eta^\bh$. In this case, we write
\[D=\bigcup_{l=1}^k (t_l=0),\]
and then the $\shD_Y[\bs]$-module $\cM_\bh^\vb$ is a $\shD_{X,D}$-lattice of $\eta^\bh_{+}\widetilde\cM$ for every $\vb\in \Z^k$. One checks immediately 
\begin{equation}\label{eq:latticegraph}
    t_l^{\pm1}\cdot \cM_\bh^\vb= \cM_\bh^{\vb\pm1_l}
\end{equation}
where $1_l\in \Z^k$ is the unit vector with the only 1 in the $l$-position.

\begin{definition}
We define 
\begin{enumerate}[label=\textup{(\roman*)}]
    \item $\cM^\vb_\bh[\bs]_{m_{\bf a}}=\cM^\vb_\bh\otimes_{\bbC[\bs]}\bbC[\bs]_{m_{\bf a}}$
    \item $\cM^\vb_\bh(\bs)=\cM^\vb_\bh\otimes_{\bbC[\bs]}\bbC(\bs)$
\end{enumerate}
where $m_{\bf a}$ is the maximal ideal of a closed point ${\bf a}\in \Spec{\bbC[\bs]}$ and $\bbC[\bs]_{m_{\bf a}}$ is the localization, and $\bbC(\bs)$ is the field of fractions of $\bbC[\bs]$. 
\end{definition} 
The above definitions are motivated by Ginsburg's ideas in \cite[\S 3.6-3.8]{Gil}.
Indeed, the case for $k=1$ is discussed  using the completion $\bbC[[\bs]]$ of $\bbC[\bs]$ with respect to the maximal ideal $m_{0}$ in \emph{loc. cit.}, while we only need the usual localization but for general $k$.

By definition, $\cM^\vb_\bh[\bs]_{m_{\bf a}}$ (resp. $\cM^\vb_\bh(\bs)$) are $\shD_Y[\bs]_{m_{\bf a}}$-modules (resp. $\shD_Y(\bs)$-modules). For an $\sA$-module $\cM$, we consider the duality functor 
\[\D(\cM)=R\mathcal{H}om_\sA(\cM, \sA)\otimes_{\sO_\bullet}\omega_\bullet[m]\]
where $\sA=\shD_\bullet,\shD_\bullet[\bs],\shD_\bullet[\bs]_{m_{\bf a}}$ or $\shD_\bullet(\bs)$ and $\bullet=Y$ or $U$. We then define the functor $j_!$ for $\sA$-modules by 
\[j_!=\D\circ j_*\circ \D;\]
in particular, for $\shD_U$-modules, $j_!$ is the usual $!$-extension of $\shD$-modules. 

We write by $Y(A)$ the variety of $Y$ over the defining ring $A$ through the base change $\bbC\to A$, where $A=\bbC[\bs], \bbC[\bs]_{m_{\vec0}}$ or $\bbC(\bs)$. Then  $\shD_Y[\bs]\otimes_{\bbC[\bs]} A$-modules are $\shD_{Y(A)}$-modules over the variety $Y(A)$.

Using Maisonobe's results in \cite{Mai}, we can prove the following duality property analogous to that of holonomic $\shD$-modules.

\begin{theorem}\label{thm:dualvan}
With notations as above, there exists a proper algebraic set $Z\subsetneq \bbC^k$ so that for every ${\bf a}\notin Z$ we have
\[\D(\cM)\stackrel{q.i.}{\simeq} \mathcal Ext^m_{\shD_{Y(A)}}(\cM,\shD_{Y(A)})\otimes_\sO\omega_Y,\]
where $\cM=$ $\cM^\vb_\bh\otimes_{\bbC[\bs]}A$ for $A=\bbC[\bs]_{m_{\bf a}}$ and $ \bbC(\bs)$.
\end{theorem}
\begin{proof}
By \cite[R\'esultat 1]{Mai}, we know the relative characteristic variety of $\cM^\vb_\bh$ is $\Lambda\times \bbC^k$, where $\Lambda$ is a conic Lagrangian in $T^*Y$. By definition, one can check that taking characteristic varieties and taking localization commute. Hence, the $\shD_{Y(\bbC(\bs))}$-characteristic variety of $\cM^\vb_\bh(\bs)$ is
\[\Lambda_{\bbC(\bs)}\coloneqq\Lambda\times \bbC^k\times_{\bbC^k}\Spec{\bbC(\bs)}.\]
This means that $\cM^\vb_\bh(\bs)$ is $\shD_{Y(\bs)}$-holonomic. Therefore, the case for $A=\bbC(\bs)$ follows by for instance \cite[Theorem D.4.3.]{HTT} and the Bernstein inequality (the Bernstein inequality is true for $\shD$-modules over fields of characteristic 0).  

When $A=\bbC[\bs]$, since the Bernstein inequality does not hold for coherent $\shD_{Y(A)}$-modules (see \cite[\S 2]{Mai} for more details), we cannot apply the above argument directly. We take a filtered free resolution of $\cM^\vb_\bh$ as a relative $\shD$-module over $\bbC^k$. Then the standard spectral sequence argument implies that  $\gr^\rel_\bullet(\mathcal Ext^l_{\shD_{Y(\bbC[\bs])}}(\cM^\vb_\bh,\shD_{Y(\bbC[\bs])}))$ is a subquotient of  
$$\mathcal Ext^l_{\gr_\bullet\shD_{Y(\bbC[\bs])}}(\gr^\rel_\bullet(\cM^\vb_\bh),\gr_\bullet\shD_{Y(\bbC[\bs])})$$
for every $l$,where by definition $\gr_\bullet\shD_{Y(\bbC[\bs])}=\gr_\bullet \shD_Y\otimes_\bbC\bbC[\bs]$.

By \cite[Proposition 14]{Mai} and \cite[A.IV Theorem 4.10]{Bj}, we have $j(\gr^\rel_\bullet(\cM^\vb_\bh))=m$ and 
$$\mathcal Ext^l_{\shD_{Y(\bbC[\bs])}}(\cM^\vb_\bh,\shD_{Y(\bbC[\bs])})=0$$
for $l<m$. Since $$\supp(\gr^\rel_\bullet(\cM^\vb_\bh))=\Lambda\times \bbC^k$$
by the definition of relative characteristic varieties,
$\gr^\rel_\bullet(\mathcal Ext^l_{\shD_{Y(\bbC[\bs])}}(\cM^\vb_\bh,\shD_{Y(\bbC[\bs])}))$ is supported on $\Lambda\times\bbC^k$. Hence $\mathcal Ext^l_{\shD_{Y(\bbC[\bs])}}(\cM^\vb_\bh,\shD_{Y(\bbC[\bs])})\otimes_\sO\omega_Y$ is major\'e by $\Lambda$ in the sense of Maisonobe (\cite[D\'efinition 1]{Mai}). Then by \cite[Proposition 8]{Mai}, we conclude that its relative characteristic variety  is 
\[\bigcup_\alpha \Lambda_\alpha\times S^l_\alpha\]
where $\Lambda_\alpha$ is a Lagrangian supported on $\Lambda$ for every $\alpha$ and an algebraic subset $S^l_\alpha\subset \bbC^k$. 

Since the dimension of the support of 
$$\mathcal Ext^l_{\gr_\bullet\shD_{Y(\bbC[\bs])}}(\gr^\rel_\bullet(\cM^\vb_\bh),\gr_\bullet\shD_{Y(\bbC[\bs])})$$ 
is $< m+k$ for $l>m$ (see for instance \cite[Theorem D.4.4]{HTT}), $S^l_\alpha$ is a proper algebraic subset of $\bbC^k$ for every $\alpha$ and for $l>m$. Therefore, by \cite[Proposition 9]{Mai}, the $\bbC[\bs]$-module support of $$\mathcal Ext^l_{\shD_{Y(\bbC[\bs])}}(\cM^\vb_\bh,\shD_{Y(\bbC[\bs])})\otimes_\sO\omega_Y$$ 
is 
\[\bigcup_\alpha S_\alpha^l\subsetneq{\bbC^k}\]
for $l>m$. 

Now we take 
$$Z=\bigcup_{\alpha, l>m}S_\alpha^l$$
which is a proper algebraic subset of $\bbC^k$ and obtain that \[\mathcal Ext^l_{\shD_{Y(\bbC[\bs])}}(\cM^\vb_\bh,\shD_{Y(\bbC[\bs])})\otimes_{\bbC[\bs]}\bbC[\bs]_{m_{\bf a}}=\mathcal Ext^l_{\shD_{Y(\bbC[\bs]_{m_{\bf a}})}}(\cM^\vb_\bh[\bs]_{m_{\bf a}},\shD_{Y(\bbC[\bs]_{m_{\bf a}})})=0\]
for $\alpha\notin Z$ and for $l>m$. Therefore, the case for $A=\bbC[\bs]_{m_{\bf a}}$ also follows. 
\end{proof}

\begin{theorem}\label{thm:bblocal}
 \begin{enumerate}[label=\textup{(\roman*)}]
    \item For any $\vb\in \Z^k$, $$\cM^\vb_h(\bs)=j_*(\bh^{\bs}\cdot\cM_U(\bs))=j_!(\bh^{\bs}\cdot\cM_U(\bs))=j_{!*}(\bh^\bs\cdot\cM_U(\bs));$$\label{item:bblocal1}
    \item For $\vb\in \Z^k$ with $v_l\gg0$ for every $l$, $\cM^{-\vb}_h[\bs]_{m_{\vec0}}=j_*(\bh^{\bs}\cdot\cM_U[\bs]_{m_{\vec0}})$;\label{item:bblocal}
    \item For $\vb\in \Z^k$ with $v_l\gg0$ for every $l$, 
    $$\cM^\vb_h[\bs]_{m_{\vec0}}=j_!(\bh^{\bs}\cdot\cM_U[\bs]_{m_{\vec0}})=j_{!*}(\bh^{\bs}\cdot\cM_U[\bs]_{m_{\vec0}}).$$
\end{enumerate}
\end{theorem}
\begin{proof}
The strategy of the proof is similar to that of \cite[Theorem 3.8.1, Proposition 3.8.3 and Corollary 3.8.4]{Gil} by applying generalized $b$-functions for lattices as in Theorem \ref{thm:bflattice}.  


As we identify $s_l$ with $-t_l\partial_{t_l}$, by Theorem \ref{thm:bflattice}, there exists a generalized $b$-function $b_\vb(\bs)$ so that 
\[b_\vb(\bs)\cM^\vb_\bh\subseteq \cM^{\vb+\vec1}_\bh\]
where $\vec1=(1,1,\dots,1)\in \Z^k$, thanks to \eqref{eq:latticegraph} again. Since $b_\vb(\bs)$ is invertible in $\bbC(\bs)$, we conclude that 
\be\label{eq:const}
\cM^\vb_\bh(\bs)=\cM^{\vb+\bq}_\bh(\bs)
\ee 
for every $\bq=(q,q,\dots,q)\in \Z^k$ and hence 
\be\label{eq:const1}
\cM^\vb_\bh(\bs)=j_*(\bh^{\bs}\cdot\cM_U(\bs))
\ee
for every $\vb\in \Z^k$. 

By Theorem \ref{thm:dualvan} for $\D(\cM_U)$ in the case for $A=\bbC(\bs)$, we know that 
$j_!(\bh^{\bs}\cdot\cM_U(\bs))$
is a coherent $\shD_{Y(\bs)}$-module instead of a complex. One can easily check 
\[j_!(\bh^{\bs}\cdot\cM_U(\bs))|_U\simeq\bh^{\bs}\cdot\cM_U(\bs).\]
By adjunction we hence have a natural morphism 
\[j_!(\bh^{\bs}\cdot\cM_U(\bs))\rarr j_*(\bh^{\bs}\cdot\cM_U(\bs))\]
and we define $j_{!*}(\bh^{\bs}\cdot\cM_U(\bs))$ to be its image. By duality, $j_{!*}(\bh^{\bs}\cdot\cM_U(\bs))$ is the minimal extension of $\bh^{\bs}\cdot\cM_U(\bs)$.

By minimality, we have for every $\vb\in \Z^k$
\[j_{!*}(\bh^{\bs}\cdot\cM_U(\bs))\hookrightarrow \cM^\vb_\bh(\bs)\] 
and the quotient is supported on $Y\setminus U$. By coherence and nullstellensatz, we see that if $q\gg0$, 
\[\cM^{\vb+\bq}_\bh(\bs)\hookrightarrow j_{!*}(\bh^{\bs}\cdot\cM_U(\bs))\] 
and hence by \eqref{eq:const} for every $\vb\in\Z^k$
\be\label{eq:const2}
\cM^\vb_\bh(\bs)=j_{!*}(\bh^{\bs}\cdot\cM_U(\bs)).
\ee 

To prove the second equality in the first statement, we use duality. By \eqref{eq:const1} and \eqref{eq:const2} the natural morphism 
\be \label{eq:const3}
j_!(\bh^{\bs}\cdot\cM_U(\bs))\rarr j_{*}(\bh^\bs\cdot\cM_U(\bs))
\ee 
is surjective, and we denote the kernel by $K$, which is a holonomic $\shD_{Y(\bbC(\bs))}$-module. It is clear that 
\[\D(\bh^\bs\cdot\cM_U[\bs])\simeq \bh^\bs\cdot\D(\cM_U)[\bs] \textup{ and }\D(\bh^\bs\cdot\cM_U(\bs))\simeq \bh^\bs\cdot\D(\cM_U)(\bs).\]
Since $\D\circ\D$ is identity, we then have 
\[\D(j_*(\bh^{\bs}\cdot\cM_U(\bs)))=j_!(\bh^{\bs}\cdot\D(\cM_U)(\bs)))\textup{ and }\D(j_!(\bh^{\bs}\cdot\cM_U(\bs)))=j_*(\bh^{\bs}\cdot\D(\cM_U)(\bs))).\]
Therefore, we have an exact sequence 
\[0\rarr j_!(\bh^{\bs}\cdot\D(\cM_U)(\bs))\rarr j_*(\bh^{\bs}\cdot\D(\cM_U)(\bs))\rarr \D K\rarr 0.\]
Replacing $\cM_U$ by $\D\cM_U$, we have 
\[j_*(\bh^{\bs}\cdot\D(\cM_U)(\bs))=j_{!*}(\bh^{\bs}\cdot\D(\cM_U)(\bs))\]
and hence $\D K$ and $K$ are both 0. Therefore, the morphism \eqref{eq:const3} is identity.  

One observes that for $\vb\in\Z^k$ with $v_l\gg0$ for every $l$, the $b$-functions $b_\vb(\bs)$ and $b_{-\vb}(\bs)$ are invertible in $\bbC[\bs]_{m_{\vec0}}$ and the second statement follows. 

Now we apply Theorem \ref{thm:dualvan} for $\cM^{-\vb}_\bh$ and obtain a proper algebraic subset $Z$.
Sine $\Z^k$ is dense in $\bbC^r$ (with respect to the Zariski-topology), we conclude that $\vec 0\notin Z$ for certain $v_l\gg0$ with $l=1,\dots k$. Therefore 
\[j_!(\bh^{\bs}\cdot\cM_U[\bs]_{m_{\vec0}})\]
is also a coherent $\shD_{Y(\bbC[\bs]_{m_{\vec 0}})}$-module instead of a complex. Hence we can define $j_{!*}(\bh^{\bs}\cdot\cM_U[\bs]_{m_{\vec0}})$ to be the minimal extension of $\bh^{\bs}\cdot\cM_U[\bs]_{m_{\vec0}}$ similar to how we define $j_{!*}(\bh^{\bs}\cdot\D(\cM_U)(\bs))$.

We know that $j_*(\bh^\bs\cdot \cM_U(\bs))$ is the localization of $j_*(\bh^\bs\cdot \cM_U[\bs]_{m_{\vec0}})$ at the generic point of $\Spec\bbC[\bs]_{m_{\vec0}}$. One can easily check that the duality functor $\D$ commutes with localization. Hence, we have 
\be\label{eq:const4}
   \D(j_*(\bh^\bs\cdot \cM_U(\bs)))=\D(j_*(\bh^\bs\cdot \cM_U[\bs]_{m_{\vec0}}))\otimes_{\bbC[\bs]_{m_{\vec0}}}\bbC(\bs).
\ee
We also have a commutative diagram
\be\label{eq:j_!*cmd}
\begin{tikzcd}
j_!(\bh^\bs\cdot \cM_U[\bs]_{m_{\vec0}})\arrow[r]\arrow[d,hook]&j_*(\bh^\bs\cdot \cM_U[\bs]_{m_{\vec0}})\arrow[d,hook]\\
j_!(\bh^\bs\cdot \cM_U(\bs))\arrow[r,"="]&j_*(\bh^\bs\cdot \cM_U(\bs)).
\end{tikzcd}
\ee
The second vertical morphism is injective as $j_*$ is exact. By applying \eqref{eq:const4} for $\D\cM_U$, to conclude that the first vertical morphism is injective, it is enough to prove that the morphism given by multiplication by $b_1(\bs)$
\[j_!(\bh^\bs\cdot \cM_U[\bs]_{m_{\vec0}})\xrightarrow{\cdot b_1(\bs)}j_!(\bh^\bs\cdot \cM_U[\bs]_{m_{\vec0}})\]
is injective for every polynomial $b_1(\bs)\not\in m_{\vec 0}$. Indeed, if on the contrary the morphism 
\[j_!(\bh^\bs\cdot \cM_U[\bs]_{m_{\vec0}})\xrightarrow{\cdot b_1(\bs)}j_!(\bh^\bs\cdot \cM_U[\bs]_{m_{\vec0}})\]
has a non-zero kernel $\cK$, then $\cK$ is major\'e by a Langrangian over $\bbC[\bs]_{m_{\vec 0}}$. Since $\cK$ is killed by $b_1(\bs)$, the relative characteristic variety of $\cK$ has dimension $<n+k$ and hence the graded number $j(\cK)>n$ (see \cite[Definition A.IV.1.8]{Bj}). But by \cite[Proposition A.IV 2.6]{Bj}, $j_!(\bh^\bs\cdot \cM_U[\bs]_{m_{\vec0}})$ is $n$-pure over $\shD_X[\bs]_{m_{\vec0}}$ (thanks to Theorem \ref{thm:dualvan} again) and hence $j(\cK)=n$, which is a contradiction. Therefore, the first horizontal morphism in Diagram \eqref{eq:j_!*cmd} is injective, and we conclude that \[j_!(\bh^{\bs}\cdot\cM_U[\bs]_{m_{\vec0}})=j_{!*}(\bh^{\bs}\cdot\cM_U[\bs]_{m_{\vec0}}).\]
Running the argument in proving \eqref{eq:const3}, we also obtain 
\[\cM^\vb_h[\bs]_{m_{\vec0}}=j_{!*}(\bh^{\bs}\cdot\cM_U[\bs]_{m_{\vec0}})\]
for $\vb\in \Z^k$ with $v_l\gg0$ for every $l$.
\end{proof}

\begin{coro}[$\Rightarrow$Theorem \ref{thm:introex1}+Theorem \ref{thm:introex2}]\label{cor:j*j!}\label{cor:mainc}
For $\vb\in \Z^k$ with $v_l\gg0$ for every $l$, we have quasi-isomorphisms
\[\bbC\stackrel{\bL}{\otimes}_{\bbC[\bs]}\cM^{-\vb}_\bh\simeq j_*\cM_U \textup{ and } \bbC\stackrel{\bL}{\otimes}_{\bbC[\bs]}\cM^\vb_\bh\simeq j_!\cM_U.\]
In particular, for $\vb\in \Z^k$ with $v_l\gg0$ for every $l$
\[\frac{\cM_\bh^{-\vb}}{(s_1,\dots,s_k)\cM_\bh^{-\vb}}\simeq \shD_Y\cdot(\prod_lh_l^{-v_l}\cM_0)=j_*\cM_U\]
and 
\[\frac{\cM_\bh^{\vb}}{(s_1,\dots,s_k)\cM_\bh^{\vb}}\simeq j_!\cM_U\textup{ and }j_{!*}\cM_U=\shD_Y\cdot(\prod_lh_l^{v_l}\cM_0).\]
\end{coro}

\begin{proof}
It is obvious that 
\[\bbC\stackrel{\bL}{\otimes}_{\bbC[\bs]_{m_{\vec0}}}\bh^\bs\cdot\cM_U[\bs]_{m_{\vec0}}\simeq \cM_U.\]
As $j_*$ is exact, we have 
\begin{equation}\label{eq:j*}
    \bbC\stackrel{\bL}{\otimes}_{\bbC[\bs]_{m_{\vec0}}}j_*(\bh^\bs\cdot\cM_U[\bs]_{m_{\vec0}})\simeq j_*\cM_U.
\end{equation}
Since 
\[\bbC\stackrel{\bL}{\otimes}_{\bbC[\bs]}\cM^{-\vb}_\bh\simeq\bbC\stackrel{\bL}{\otimes}_{\bbC[\bs]_{m_{\vec0}}}\cM^{-\vb}_\bh(m_{\vec0}),\]
using Thoerem \ref{thm:bblocal} \ref{item:bblocal} and \eqref{eq:j*}, we get for $\vb\in \Z^k$ with $v_l\gg0$ for every $l$ $$\bbC\stackrel{\bL}{\otimes}_{\bbC[\bs]}\cM^{-\vb}_\bh\simeq j_*\cM_U.$$ 
One can check that the functors $\bbC\stackrel{\bL}{\otimes}_{\bbC[\bs]}\bullet$ and $\D$ commute. Hence, we have for $\vb\in \Z^k$ with $v_l\gg0$ for every $l$ 
\[\bbC\stackrel{\bL}{\otimes}_{\bbC[\bs]_{m_{\vec0}}}\cM^{\vb}_\bh[\bs]_{m_{\vec0}}\simeq\D(j_*(\bh^\bs\cdot(\D\cM_U)[\bs]_{m_{\vec0}})\stackrel{\bL}{\otimes}_{\bbC[\bs]_{m_{\vec0}}}\bbC)\simeq\D(j_*(\D\cM_U))=j_!\cM_U.\]
Therefore we obtain
for $\vb\in \Z^k$ with $v_l\gg0$ for every $l$ $$\bbC\stackrel{\bL}{\otimes}_{\bbC[\bs]}\cM^{\vb}_\bh\simeq j_!\cM_U.$$ 
We have proved the first statement.

Using the Koszul resolution of $\bbC$ as a $\bbC[\bs]$-module, the first statement implies 
\[\frac{\cM_\bh^{-\vb}}{(s_1,\dots,s_k)\cM_\bh^{-\vb}}\simeq j_*\cM_U \textup{ and }\frac{\cM_\bh^{\vb}}{(s_1,\dots,s_k)\cM_\bh^{\vb}}\simeq j_!\cM_U.\]
Applying the argument of the proof of the equation \eqref{eq:const1}, one obtains $$\shD_Y\cdot(\prod_lh_l^{-v_l}\cM_0)=j_*\cM_U.$$ The proof of $\shD_Y\cdot(\prod_lh_l^{v_l}\cM_0)=j_{!*}\cM_U$ is similar to that of \eqref{eq:const2} (see also the proof of Lemma 3.8.2 in \cite{Gil}).
\end{proof}

\section{Index theorem for lattices}\label{sec:logind}

For regular holonomic $\shD_U$-module, using the microlocalization of the sheaf of logarithmic differential operators, Ginsburg \cite[Appendix A.]{Gil1} proved the following deep theorem: 
\begin{theorem}[Ginsburg]\label{thm:Gil}
If $\cM_U$ is a regular holonomic $\shD_U$-module, then for every lattice $\cM$ of $\cM_U$ we have 
\[\overline\SiS\cM_U=\SiS\cM,\]
where $\overline\SiS\cM_U$ is the closure of $\SiS\cM_U$ inside $T^*(X,D)$.
\end{theorem}
Suppose that $f\colon (X,D)\to (Y,E)$ is a morphism of smooth log pairs with $\dim X=n$ and $\dim Y=m$. We consider again the diagram of the log Lagrangian correspondence of $f$
\[T^*(X,D)\stackrel{p_1}{\longleftarrow}X\times_YT^*(Y,E)\stackrel{p_2}{\longrightarrow} T^*(Y,E).\]
Analogous to the functor $f_\sharp$, we define a morphism between Chow rings 
\[f_\sharp\colon A^\bullet_{T^*(X,D)}\rarr A^\bullet_{T^*(Y,E)}\]
by 
\[f_\sharp(\bullet)=p_{2*}(p_1^*(\bullet)),\]
where $A^\bullet_{T^*(X,D)}$ and $A^\bullet_{T^*(Y,E)}$ are Chow rings of $T^*(X,D)$ and $T^*(Y,E)$ respectively. 

We first recall some preliminaries of the intersection theory. The Riemann-Roch morphism $\tau$ is 
\[\tau\colon K_\coh(\sO_-)\rarr A^\bullet_-\]
given by $\tau_-(\bullet)=\ch(\bullet).\td(\shTA_-)$, where $-$ represents a complex algebraic variety and $\ch$ denotes the Chern character and $\td(\shTA_-)$ the Todd class of the tangent sheaf.  Let us list some Riemann-Roch formulas that are needed (see \cite[Chapter 15.]{Ful}):
\begin{enumerate}[label=\textup{(\roman*)}]
    \item $g_*\circ\tau\simeq \tau\circ Rg_*$, for every proper morphism $g$;\label{item:grr1}
    \item $\tau\circ Lh^*\simeq \td([T_h])\cdot(h^*\circ\tau)$ for every local complete intersection morphism $h$, where $[T_h]$ is the virtual tangent bundle of $h$; \label{item:grr2}
    \item $\tau(\beta \otimes \alpha)\simeq \ch(\beta)\cdot\tau(\alpha)$ for $\alpha\in K_\coh(\sO_-)$ and $\beta$ is a class of a vector bundle.\label{item:grr3}
\end{enumerate}
\begin{lemma}\label{lem:rrm}
For every lattice $\cM$ of a regular holonomic $\shD_U$-module,
we have 
$$\tau([\widetilde\gr^F_\bullet\cM]\otimes [\omega_f])=[\SiS\cM].$$
\end{lemma}
\begin{proof} 
We know that $K_\coh(\sO_{-})$ has a decreasing filtration $F^\bullet$ by the codimension of supports. In the case of the claim, by the construction of multiplicity, we know  
\[[\widetilde\gr^F_\bullet\cM]\equiv \sum_pm_p\sO_{\bar p} \mod F^{n+1}\]
where $p$ goes over the generic points of the support of $\widetilde\gr^F_\bullet\cM$ and $m_p$ the multiplicity (by Theorem \ref{thm:Gil}, $\SiS\cM$ is of pure codimension $n$). Also, we  have 
\[\tau(\sO_{\bar p})\equiv \bar p \mod A^{> n}_{T^*{X,E}}=\bigoplus_{m>n}A^{m}_{T^*{(X,E)}}. \]
Since $\tau$ is compatible with the filtration $F^\bullet$, by the Riemann-Roch formula \ref{item:grr3} we conclude that 
\[\tau([\widetilde\gr^F_\bullet\cM]\otimes [\omega_f])=[\SiS\cM] \mod A^{> n}_{T^*(X,E)}.\]
Since $T^*(X,E)$ is an affine bundle with fiber dimension $n$, $A^{> n}_{T^*{(X,E)}}$ is trivial and the proof is finished. 
\end{proof}

\begin{theorem}\label{thm:GIT}
Let $(X,D)$ be a projective log smooth pair with $U=X\setminus D$. Assume that $\cM_U$ is a regular holonomic $\shD_{U}$-module and $\cM$ a lattice. Then we have 
\[\chi(X,\DR_D(\cM))=[\SiS\cM]\cdot[T^*_X(X,D)].\]
\end{theorem}
\begin{proof}
We consider the log Lagrangian correspondence for the constant map $f\colon (X,D)\to \Spec\textup{ } \bbC$:
\[T^*(X,D)\stackrel{p_1}{\longleftarrow}X\stackrel{p_2}{\longrightarrow}\Spec{\textup{ }}\bbC.\]
In this case, the Riemann-Roch morphism over $\Spec\textup{ }\bbC$ is just taking rank of vector spaces. By Corollary \ref{cor:imageK} we hence have 
\begin{equation}\label{eq:rr1}
    \tau(p_{2*}p_1^*([\widetilde\gr \cM]\otimes[\omega_f]))=f_\sharp[\widetilde\gr \cM]=\sum_i (-1)^i h^i(f_+\cM).
\end{equation}

Since $p_1$ is identified with the closed embedding $T^*_X(X,D)\to T^*(X,D)$, by Lemma \ref{lem:rrm} we have 
\[p_1^*\tau([\widetilde\gr \cM]\otimes[\omega_f]))=p_1^*[\SiS \cM]=[\SiS\cM\cdot T^*_X(X,D)]\]
in $A^n_X$. Since $[\SiS\cM\cdot T^*_X(X,D)]$ is a zero-cycle on $X$, by the Riemann-Roch formulas \ref{item:grr1} and \ref{item:grr2}, we have
\begin{equation}\label{eq:rr2}
p_{2*}p_1^*[\SiS \cM]=[\SiS\cM]\cdot[T^*_X(X,D)],
\end{equation}
the degree of the zero cycle $[\SiS\cM\cdot T^*_X(X,D)]$ in $X\simeq T^*_X(X,D)$.

Finally, by Lemma \ref{lem:resomega}, we have a resolution
$$\DR_D(\shD_{X,D})\to \omega_X(D)=\omega_f,$$ 
and hence we get
\begin{equation}\label{eq:rr3}
f_+\cM\simeq Rf_{*}\DR_D(\cM).
\end{equation}
The proof is done by combining \eqref{eq:rr1}, \eqref{eq:rr2} and \eqref{eq:rr3}. 
\end{proof}


\section{Logarithmic deformations of characteristic cycles and a non-compact Riemann-Roch theorem}\label{sec:noncrr}

\subsection{Characteristic cycle of $\shD_{X,D}[\bs](\bff^\bs\cdot\cM)$ for rational functions}\label{sect:crf}

Suppose that $(X,D)$ is a smooth log pair and $\bff=(f_1,\dots,f_k)$ is a $k$-tuple of rational functions on $X$ satisfying that the divisor of $f_l$ is supported on $D$ for each $l=1,\dots,k$.
Let $\cM_{U}$ be a regular holonomic $\shD_{U}$-module and $\cM$ be a $\shD_{X,D}$-lattice. 

We introduce independent variables $\bs=(s_1,\dots,s_k)$. Since the divisor of $f_l$ is supported on $D$, we consider the $\shD_{X,D}[\bs]$-module
\[\shD_{X,D}[\bs]( \bff^\bs\cdot\cM)\subseteq \bff^\bs\cdot j_*\cM_U[\bs].\]

By assigning the extra variables $s_l$ of order 1, the order filtration $F_\bullet$ on $\shD_{X,D}$ induces an order filtration $F_\bullet$ on $\shD_{X,D}[\bs]$ so that 
\[\grf\shD_{X,D}[\bs]=(\grf\shD_{X,D})[\bs].\]
We then can further identify $\grf\shD_{X,D}[\bs]$ with rings of functions on $T^*(X,D)\times\bbC^k$:
\[\grf\shD_{X,D}[\bs]=\pi_*\sO_{T^*(X,D)}[\bs].\]
We then can define the characteristic cycle of $\shD_{X,D}[\bs](\bff^\bs\cdot\cM)$ similar to that of coherent $\shD_{X,D}$-modules, denoted by 
\[\SiS(\shD_{X,D}[\bs](\bff^\bs\cdot\cM)).\]
It is a conic cycle in $T^*(X,D)\times\bbC^k$. 

Following Kashiwara \cite{KasBf} and Ginsburg \cite[\S2.2]{Gil}, we define the log analogue of $\Lambda^\sharp$. For a conic cycle $\Lambda$ of $T^*(X,D)$ with $\dim \Lambda=n$, since $d\log f_l(x)$ is a section of $T^*(X,D)$ for each $l$, we define the $n+k$-dimensional cycle $\Lambda^\sharp_\bff\subset T^*(X,D)\times \bbC^k$:
\[\Lambda^\sharp_\bff=\{(\xi+\sum_{l=1}^k s_l\cdot d\log f_l(x),\bs)|\xi\in \Lambda, \pi(\xi)=x, \bs=(s_1,\dots,s_k)\in \bbC^k\}.\]

\begin{theorem}\label{thm:logcycle}
With notations as above, for every lattice $\cM$, we have 
\[\SiS(\shD_{X,D}[\bs](\bff^{\bs}\cdot\cM))=\SiS^\sharp_{\bff}\cM\]
in $T^*(X,D)\times \bbC^k$.
\end{theorem}
Then we obtain an algebraic family of log-Lagrangian subvarieties $$p_2\colon\SiS^\sharp_{\bff}\cM\to \bbC^k,$$ with the central fiber at $0$ a $n$-dimensional conic cycle. Hence, $p_2$ gives a deformation of $\overline\SiS\cM_U$ in the logarithmic cotangent bundle $T^*(X,D)$ (by Theorem \ref{thm:Gil}).

\begin{proof}[Proof of Theorem \ref{thm:logcycle}]
Since the divisor of $f_l$ is supported on $D$, we know that for $P(s)\in \shD_{X,D}[\bs]$
\[P(\bs)\mapsto\bff^{-\bs}\cdot P(\bs)\cdot \bff^\bs\]
defines an automorphism $\eta_\bff$ of $\shD_{X,D}[s]$ and it induces an isomorphism 
\[\eta_\cM\colon\shD_{X,D}[\bs](\bff^{\bs}\cdot\cM)\simeq \cM[\bs]\]
by 
\[\eta_\cM(\bff^\bs \cdot u)=u\]
compatible with $\eta_\bff$.
The isomorphism $\eta_\bff$ induces an isomorphism 
$$\eta_\bff\colon T^*(X,D)\times \bbC^k\to T^*(X,D)\times \bbC^k$$
by 
\[(x,\xi,\bs)\mapsto (x,  \xi+\sum_{l=1}^ks_ld\log f_l,\bs).\]
We then have
\begin{equation}\label{eq:eta111}
    \eta_\bff(\SiS\cM)=\SiS^\sharp_\bff\cM.
\end{equation}

Now we fix a coherent filtration $(\cM,F_\bullet)$ and define a filtration for $\shD_{X,D}[\bs](\bff^{\bs}\cdot\cM)$ by 
\[F_p(\shD_{X,D}[\bs](\bff^{\bs}\cdot\cM))=\sum_{|{\bf i}|+j=p}\bs^{\bf i}\cdot \bff^\bs\cdot F_j\cM,\]
where $\bs^{\bf i}=s_1^{i_1}\cdot s_2^{i_2}\cdots s_k^{i_k}$. We also define a filtration on $\cM[\bs]$ by
\[F_p\cM[\bs]=\sum_{|{\bf i}|+j=p}\bs^{\bf i}\cdot F_j\cM.\]
Since $\eta_\bff$ preserves the order filtration of $\shD_{X,D}[s]$, $\eta_\cM$ is also a filtration-preserving isomorphism. 
Hence, we obtain that 
\[\SiS(\shD_{X,D}[\bs](\bff^{\bs}\cdot\cM))=\supp(\gr^F_\bullet(\shD_{X,D}[\bs](\bff^{\bs}\cdot\cM))=\eta_\bff(\SiS(\cM)\times\bbC^r).\]
The proof is now done by \eqref{eq:eta111}.
\end{proof}

\begin{proof}[Proof of Theorem \ref{thm:openRR}]
 We let $\cM_U$ be the regular holonomic $\shD_U$-module of $\sF^\bullet$ under Riemann-Hilbert correspondence and let $\cM$ be a $\shD_{X,D}$-lattice. For simplicity, we write $\Lambda=\overline\SiS(\sF^\bullet)$, which is also the characteristic cycle of $\cM$. Then the family $p_2\colon\Lambda_f^\sharp\to \bbC$ gives a deformation of the central fiber $\Lambda$ inside $T^*(X,D)$. 

By Theorem \ref{thm:logind}, we know that 
\[\chi(U, \cF^\bullet)=\chi(X,\DR_D(\cM))=[\Lambda]\cdot[T^*_X(X,D)]. \]
Since $p_2$ is a deformation of $\Lambda$, we know for every $0\not= s_0\in \bbC$
\[[\Lambda]\cdot[T^*_X(X,D)]=[(\Lambda_f^\sharp)_{s_0}]\cdot[T^*_X(X,D)].\]
But since $\Lambda$ is conic, we know 
\[[(\Lambda_f^\sharp)_{s_0}]\cdot[T^*_X(X,D)]=\sum_{v}n_v\gdeg_f^{\log}(\Lambda_v).\]
\end{proof}

\section{An alternative proof of Theorem \ref{thm:logind}}\label{sec:rrtop}

In this section, we will work with the real cotangent bundles $T^*U$ and $T^*(X, D)$ that underly their complex counterparts.

First, we recall a result of Kashiwara for the Euler characteristic on a non-compact set \cite{Kash-index}. 

\begin{theorem}[Kashiwara]\label{thm:Kasind}
Let $U$ be a complex $n$-dimensional quasi-projective variety. Let $\sF$ be a $\R$-constructible sheaf on it. 
Let $\varphi: U \to \R$ be a $C^2$-function. Set $\Gamma_{d \varphi} = \{ d\varphi(x): x \in U\} \In T^*U$. We assume that ${\rm supp} \, F \cap  \{\varphi(x) \leq t \}$ is compact for any $t$ and $\SiS(\sF) \cap \Gamma_{d\varphi}$ is compact. Then we have 
\[ \dim H^j(U; \sF) < \infty \]
for any $j$ and
\[ \chi(U, \sF) = (-1)^{n} \SiS(\sF) \cdot \Gamma_{d \varphi}. \]
\end{theorem}
Note the extra $(-1)^n$ is due to the degree shift between constructible sheaf and perverse sheaf.




In this paper, we work not in $T^*U$ but in its log compactification $T^*(X, D)$. We shall choose a smooth real perturbation of zero section $\Gamma_{d\varphi} \In T^*U$ such that it extends to a smooth section $\wb{\Gamma}_{d\varphi}$ in $T^*(X, D)$. This $ \wb{\Gamma}_{d\varphi}$  serves as a perturbation of zero-section $T^*_X(X, D)$. 

\subsection{Real log cotangent bundle and perturbation of zero section}
First, we recall that every complex rank $n$ vector bundle $E \to X$ has an underlying real rank $2n$ vector bundle, denoted as $E_\R$. And since intersection of cycles are topological notion, we may perform smooth real perturbation to the zero section, instead of algebraic or holomorphic ones. 

Pick a point $z$ on the snc divisor $D$. Assume there are $r$ irreducible components intersects at $z$. We choose local complex coordinates centered at $z$, denoted as $x_1, \cdots, x_r, y_1, \cdots, y_k$, such that such that $D = \{x_1\cdots x_r=0\}$. The complex vector bundle $T^*(X, D)$ then has a local frame 
\[d \log x_1, \cdots, d \log x_r, d y_1, \cdots, dy_k. \]

For the underlying real vector bundle $T^*(X, D)_\R$, we have obtain the local frame by taking the real and imaginary parts of the above holomorphic local frame
\begin{align} & 
d \log \rho_1, d \theta_1 , \cdots, d \log \rho_r, d \theta_r, \quad x_i = \rho_i e^{i \theta_i} \notag \\
& d y_{1,R}, dy_{1,I}, \cdots, dy_{k,R}, d y_{k, I}, \quad y_i = y_{i,R} + \sqrt{-1} y_{i, I}. \label{coord}
\end{align}

Let $H \In X$ be a smooth complex hypersurface. A function $\rho_H: X \to \R$ is a {\em real defining function for $H$}, if for any local holomorphic defining function $t_\alpha: U_\alpha \to \C$ of $H \cap U_\alpha$, we have $\rho_H|_{U_\alpha} = f_\alpha |t_\alpha|$. By a partition of unity argument, we see the real defining function for $H$ exists, and is unique up to multiplication by smooth positive functions. 

Similarly, if $D=\cup_i D_i$ is a simple normal crossing divisor, we say $\rho_D: X \to \R$ is a real defining function for $D$,  if locally $\rho_D$ equals the modulus of a complex defining function of $D$ up to multiplication by smooth positive function. Clearly, we may take $\rho_D = \prod_i \rho_{D_i}$, hence $\rho_D$ exists as well. 



Let $\rho$ be a real defining function for $D$. Then the log closure
\[ \Gamma_{d \log \rho} := \wb{\Gamma}_{d \log \rho|_U} \In T^*(X, D) \]
is a smooth section in $T^*(X, D)$. 

\begin{prop} \label{dlogrho}
Let $z \in D$, and let \eqref{coord} be local coordinates around $z$. Then 
\[ {d \log \rho}|_z = d\log \rho_1 + \cdots + d\log \rho_1 + \sum_{j=1}^k (a_j d y_{j,R} + b_j d y_{j, I}) , \quad a_j, b_j \in \R.\]
\end{prop}
\begin{proof}
Locally $\rho = f |x_1| \cdots |x_r|$, for a positive smooth function $f$. Hence 
\[ {d \log \rho}|_z = d\log \rho_1 + \cdots + d\log \rho_1 +  {d \log f}|_z.  \]
We can further verify that $d \log f|_z$ has no component in $d \log \rho_i$ and $d\theta_i$, for $i=1, \cdots, r$. 

\end{proof}

\subsection{Log closure of conic Lagrangians}
Now we can give the alternative proof of Theorem \ref{thm:logind}. We claim that
\[[\overline{\SiS}(F)] \cdot [T^*_X(X, D)] =\Gamma_{-d \log \rho} \cdot \overline {\SiS}(F) = \Gamma_{-d \log \rho} \cdot SS(F)
= \chi(U, F) \]
where the first equality follows from perturbation of the zero-section, and the last equality follows from Theorem \ref{thm:Kasind}. The middle equality is proven in the following lemma. 

\begin{lemma}
Let $\Lambda \In T^*U$ be a real analytic conic Lagrangian, and $\overline \Lambda \In T^*(X, D)$ its closure. Then 
\[ \Gamma_{-d \log \rho} \cap \overline {\Lambda} = \Gamma_{-d \log \rho} \cap \  {\Lambda}. \]
\end{lemma}
\begin{proof}
We prove this by contradiction. Suppose the intersection $\Gamma_{-d \log \rho} \cap ( \overline {\Lambda} \backslash \Lambda)$ is non-empty, and contains a point $q$ that lies over $z \in D$. Choose a coordinate patch around $z$ as in \eqref{coord}.
 
The log cotangent bundle $T^*(X, D)$ has fiber coordinates $\eta_{\rho_i}, \xi_{\theta_i}, \xi_{y,j,R}, \xi_{y,j,I}$, in which the  Liouville 1-form  can be written as 
\[ \lambda = \sum_{i=1}^r  \left( \eta_{\rho_i} d \log \rho_i + \xi_{\theta_i} d \theta_i \right) + \sum_{j=1}^k \left(\xi_{y,j,R} d y_{j,R} + \xi_{y,j,I} d y_{j,I}\right) \]
Note that since $\La$ is a conic Lagrangian,  $\lambda$ vanishes on $T\La$.  

Using the curve selection lemma, here exists a real analytic curve $\gamma: [0,1) \to T^*(X, D)$, such that $\gamma(0)=q$ and $\gamma(0,1) \In \Lambda$.   We may write $\gamma(t)$ in coordinates of $T^*(X, D)$ as $\rho_i(t),\theta_i(t), \cdots$ and $\eta_{\rho_i}(t), \xi_{\theta_i}(t), \cdots $. 

We consider the line integral  
\[ \int_{\gamma(0,\epsilon)} \lambda, \quad \epsilon \ll 1. \]
On  one hand, the integral is zero, since $\lambda$ vanishes on $d{\gamma(t)}/dt \in T\Lambda$. On the other hand, if we write in $\lambda$ in component, we have 
\begin{align*}
      \int_{\gamma(0,\epsilon)} \lambda 
     &= \int_{0}^\epsilon \sum_{i=1}^r \left( \eta_{\rho_i}(t) \frac{d \rho_i(t)}{\rho_i(t)} + \xi_{\theta_i}(t) d \theta_i(t) \right) \\
     & + \sum_{j=1}^k \left(\xi_{y,j,R}(t) d y_{j,R}(t) + \xi_{y,j,I}(t) d y_{j,I}(t)\right) \\
\end{align*}
As $t \to 0$, by Proposition \ref{dlogrho},  we have 
\[ \lim \rho_i(t)=0,  \quad \lim \eta_{\rho_i}(t) = -1, \quad \lim \xi_{\theta_i}(t) = 0,  \]
and all other limits for $\theta_i,  y_{j,I},y_{j,R}, \xi_{y,j,R}, \xi_{y,j,I} $ exist. The integral for the $d \rho / \rho$ part gives 
\[ \int_{0}^\epsilon \sum_{i=1}^r  \eta_{\rho_i}(t) \frac{d \rho_i(t)}{\rho_i(t)} \sim - \int_{0}^\epsilon \sum_{i=1}^r   \frac{d \rho_i(t)}{\rho_i(t)} \sim -\infty,\]
 whereas the other terms in the integral are bounded. Hence $\int_{\gamma(0,\epsilon)} \lambda = -\infty$ and we have a contradiction. This proves the lemma. 
\end{proof}

\bibliographystyle{amsalpha}
\bibliography{mybib}
\end{document}